\documentclass[reqno]{amsart}

\usepackage{amsfonts}
\usepackage{amssymb}

\usepackage{amscd}
\usepackage{pictexwd,dcpic}
\usepackage{graphicx}
\usepackage{tikz}
\usepackage{fullpage}
\usepackage{hyperref}
\hypersetup{
    colorlinks=true,
    linkcolor=blue,
    filecolor=magenta,
    urlcolor=cyan,
    citecolor=cyan,}
\usepackage[nameinlink,capitalize]{cleveref}
\usepackage{enumitem}

\title{Duality 
and the equations of Rees rings and tangent algebras}

\author{Matthew Weaver}
\address{Department of Mathematics, University of Notre Dame, 255 Hurley Bldg, Notre Dame, IN 46556}
\email{mweaver6@nd.edu}

\date{}	

\newtheorem{thmx}{Theorem}

\newtheorem{thm}{Theorem}[section]

\newtheorem{prop}[thm]{Proposition}
\newtheorem{lemma}[thm]{Lemma}
\newtheorem{cor}[thm]{Corollary}
\numberwithin{equation}{section}

\theoremstyle{definition}
\newtheorem{rem}[thm]{Remark}
\newtheorem{set}[thm]{Setting}
\newtheorem{notat}[thm]{Notation}
\newtheorem{defn}[thm]{Definition}
\newtheorem{quest}[thm]{Question}
\newtheorem{ex}[thm]{Example}
\newtheorem{obs}[thm]{Observation}
\newtheorem{conj}[thm]{Conjecture}

\def\A{\mathcal{A}}
\def\B{\mathcal{B}}
\def\I{\mathcal{I}}
\def\J{\mathcal{J}}

\def\K{\mathcal{K}}
\def\L{\mathcal{L}}
\def\m{\mathfrak{m}}

\def\p{\mathfrak{p}}

\def\R{\mathcal{R}}
\def\S{\mathcal{S}}
\def\F{\mathcal{F}}

\def\coker{\mathop{\rm coker}}
\def\dim{\mathop{\rm dim}}
\def\edim{\mathop{\rm edim}}
\def\fitt{\mathop{\rm Fitt}}
\def\grade{\mathop{\rm grade}}
\def\hgt{\mathop{\rm ht}}
\def\rk{\mathop{\rm rank}}

\def\spec{\mathop{\rm Spec}}

\def\chr{\mathop{\rm char}}
\def\bideg{\mathop{\rm bideg}}

\def\Hom{{\rm Hom}}
\def\id{{\rm id}}

\usepackage[all]{xy}

\newcommand*{\downdot}{_{\scalebox{0.5}{$\bullet$}}}
\newcommand*{\updot}{^{\scalebox{0.5}{$\bullet$}}}

\def\fseq{\underline{\,f\,}}

\begin{document}
\maketitle

\begin{abstract}
Let $E$ be a module of projective dimension one over a Noetherian ring $R$ and consider its Rees algebra $\R(E)$. We study this ring as a quotient of the symmetric algebra $\S(E)$ and consider the ideal $\A$ defining this quotient. In the case that $\S(E)$ is a complete intersection ring, we employ a duality between $\A$ and $\S(E)$ in order to study the Rees ring $\R(E)$ in multiple settings. In particular, when $R$ is a complete intersection ring defined by quadrics, we consider its module of K\"ahler differentials $\Omega_{R/k}$ and its associated tangent algebras.
\end{abstract}

\section{Introduction}

The objective of this paper is to study the Rees algebra of particular ideals and modules over a Noetherian ring $R$. Although our treatment of the subject is purely algebraic, much of its motivation is geometric. Indeed, the Rees ring is often called the \textit{blowup algebra} since, for $I$ an $R$-ideal, the Rees algebra $\R(I)= R\oplus It\oplus I^2t^2\oplus\cdots$ is the algebraic realization of the blowup of the affine scheme $\spec(R)$ along the subscheme $V(I)$. Moreover, this notion can be extended to Rees algebras of \textit{modules} in order to treat the case of \textit{repeated} or \textit{successive} blowups. Indeed, for $I$ and $J$ ideals of $R$, the successive blowup along the disjoint subschemes $V(I)$ and $V(J)$ corresponds to the Rees algebra of $I\oplus J$, which is of course a module and not an ideal. As such, one is obligated to study Rees rings of modules, hence we proceed in this setting and consider the Rees algebra $\R(E)$ of an $R$-module $E$. We note that, as this is the more general notion, one is always able to recover the case of Rees rings of ideals.

In addition to the blowup construction, Rees algebras of modules possess many other applications to algebraic geometry. The most notable such module is the module of \textit{K\"ahler differentials} $\Omega_{R/k}$ of an affine $k$-algebra $R$. The Rees ring $\R(\Omega_{R/k})$, and its related algebras, the symmetric algebra $\S(\Omega_{R/k})$ and the fiber ring $\F(\Omega_{R/k})$, are often called \textit{tangent algebras}. If $X\subseteq \mathbb{P}_k^n$ is a subvariety with coordinate ring $R$, the fiber ring $\F(\Omega_{R/k})$ is the homogeneous coordinate ring of the \textit{tangential variety} of $X$. Moreover, the closed fibers of the map $\spec(\S(\Omega_{R/k})) \rightarrow \spec(R)$ are precisely the Zariski tangent spaces to closed points of $\spec(R)$ \cite{SUV97}. One might also consider the \textit{Jacobian module} $\mathfrak{J}$ of $R$ and its Rees algebra $\R(\mathfrak{J})$, which is the coordinate ring of the \textit{conormal variety} of $X$.

In addition to encoding the algebraic data of blowups of affine schemes, Rees rings also serve as bihomogeneous coordinate rings of graphs of rational maps between projective varieties. As these constructions are parametric, the natural question that arises is how to find the corresponding \textit{implicit} equations of these objects. Algebraically, this correlates to expressing the Rees ring $\R(E)$ as a quotient of a polynomial ring $R[y_1,\ldots,y_n]$ over the base ring $R$. Such a description of the Rees ring provides insight into a wealth of information, and is hence a desirable form to obtain. The ideal $\J$ defining this quotient ring is appropriately called the \textit{defining ideal} of $\R(E)$. Unfortunately, a concrete description of the ideal $\J$ and its generators, the \textit{defining equations} of $\R(E)$, is difficult to attain in general. However, there has been success for Rees algebras of ideals with small codimension \cite{BM16,CPW23,KPU17,Morey96,MU96,Nguyen14,Nguyen17,Weaver23,Weaver24} and modules with small projective dimension \cite{Costantini21,CPW23,SUV03,Weaver23} in a multitude of settings.

Although it is desirable to relate the Rees ring to $R[y_1,\ldots,y_n]$ and express $\R(E)$ as a quotient of such a ring, it is often more convenient to consider a much closer algebra mapping onto $\R(E)$. The natural choice is the symmetric algebra $\S(E)$ of the module $E$ and the natural epimorphism $\S(E) \rightarrow \R(E)$ with kernel $\A$. In a sense, $\A$ is a defining ideal of the Rees ring as well, as $\R(E) \cong \S(E) /\A$. Moreover, $\A$ has the convenience of being simpler than its counterpart $\J$, and is often relatable to natural constructions. Moreover, one can always deduce information on $\J$ from $\A$.

The theme of this paper is to apply a particular duality between $\A$ and $\S(E)$, when $\S(E)$ is a complete intersection ring, in order to study the Rees ring $\R(E)$. We remark that this is not an entirely novel approach, and such a duality for complete intersections was noted originally by Jouanolou \cite{Jouanolou96}, making such an isomorphism explicit through the use of \textit{Morley forms}. This was later adapted to symmetric algebras of ideals by Kustin, Polini, and Ulrich \cite{KPU Bigraded Structures}. We introduce Jouanolou's duality for any \textit{bigraded} complete intersection ring, so that we may apply it to study Rees rings of modules in multiple settings. 

\begin{thmx}\label{Intro - Duality Theorem}
    Let $R=k[x_1,\ldots,x_d]$, $T=k[y_1,\ldots,y_n]$, and $B= R\otimes_k T$. With the bigrading on $B$ given by $\bideg x_i = (1,0)$ and $\bideg y_i = (0,1)$, let $\underline{f} = f_1,\ldots,f_d \subseteq \m=(x_1,\ldots,x_d)$ be a $B$-regular sequence of bihomogeneous elements. Assume that $\bideg f_i =(\alpha_i,\beta_i)$ and let $\S=B/\mathcal{I}$, where $\mathcal{I}=(\underline{f})$. There is a bihomogeneous isomorphism of $B$-modules
$$H_\m^0 (\S) \cong \Hom_T(\S,T)(-\delta,-\tau)$$
where $\delta = (\sum_{i=1}^d \alpha_i) -d$, $\tau =\sum_{i=1}^d \beta_i$, and $H_\m^0 (\S)$ is the zeroth local cohomology module of $\S$ with respect to $\m$.
\end{thmx}

Under certain hypotheses, we can ensure that the symmetric algebra $\S(E)$ is a complete intersection ring and the $\S(E)$-ideal defining $\R(E)$ is $\A = H_\m^0(\S(E))$. Hence we intend to apply \Cref{Intro - Duality Theorem} to the symmetric algebra in this setting. As $\A$ is bigraded, we aim to determine certain graded components using \Cref{Intro - Duality Theorem} and the resolutions of certain graded components of $\S(E)$. As $\S(E)$ is a complete intersection ring, such resolutions are readily available from the Koszul complex.

As the Koszul complex $\K\downdot$ is a bigraded $B$-resolution of $\S(E)$, each $T$-graded strand $\K_i$ is a graded $T$-resolution of $\S(E)_i  = \bigoplus_j \S(E)_{(i,j)}$. Hence from \Cref{Intro - Duality Theorem} we have isomorphisms of $T$-modules
$$\A_{\delta-i} \cong \Hom_T(\S(E)_{i},T(-\tau)) \cong \ker \sigma_1^* (-\tau)$$
for $0\leq i\leq \delta$, where $\sigma_1$ is the first map of $\K_{i}$. Hence one needs only to understand certain dual maps in order to understand the $T$-modules $\A_{\delta-i}$. In particular, we may apply the tools developed by Kim and Mukundan in \cite{KM20} to describe the kernel of this dual map when the rank of $\S(E)_i$ is small enough. We apply these techniques in a handful of settings, the first of which is modules which nonlinear presentation.

We remark that there is a significant amount of literature for Rees rings of ideals and modules with \textit{linear} presentation (see e.g. \cite{CPW23,KPU17,Morey96,MU96,Nguyen14,Nguyen17,SUV03,UV93,Weaver23,Weaver24}), i.e. the entries in a presentation matrix $\varphi$ are all linear forms. Additionally, there has been recent progress for ideals and modules with \textit{almost} linear presentation \cite{BM16,Costantini21,CHW08}, i.e. the entries of such a matrix $\varphi$ are all linear, except for one column with entries of a higher degree. In this direction, our first application of \Cref{Intro - Duality Theorem} concerns Rees algebras of ideals and modules with \textit{almost} almost linear presentation, namely when a presentation matrix $\varphi$ consists of linear entries, except for two columns of quadrics. The main result of \Cref{AALP Section} is as follows.

\begin{thmx}\label{Intro - AALP theorem}
Let $R=k[x_1,\ldots,x_d]$ with $d\geq 3$ and let $E$ be a torsion-free $R$-module with rank $e$ and projective dimension one, minimally generated by $\mu(E)=d+e$ elements and presented by matrix $\varphi$. Assume that $\varphi = [\varphi'\,|\,\varphi'']$ where $\varphi'$ has $d-2$ columns with linear entries and $\varphi''$ consist of two columns of quadrics. Assume that $I_1(\varphi) = (x_1,\ldots,x_d)$ and $E$ satisfies $G_d$. The defining ideal $\A$ of $\R(E)$ is generated as follows.
\begin{enumerate}
    \item[(i)] $\A_2$ is minimally generated by one equation of bidegree $(2,d)$.

    \item[(ii)] If $\hgt I(\sigma_1) \geq 3$ then $\A_1$ is minimally generated by $d$ equations of bidegree $(1,2d-2)$. 

    \item[(iii)]  If $\hgt I(\rho_1) \geq 2$ then $\A_0$ is minimally generated by one equation of bidegree $(0,4d-4)$.
\end{enumerate}
Here $\sigma_1$ and $\rho_1$ are the first maps in the graded strands $\K_1$ and $\K_2$, resolving $\S(E)_1$ and $\S(E)_2$ respectively. 
\end{thmx}

Whereas \Cref{Intro - AALP theorem} does not say what the equations of $\A$ are, it does provide certain criteria for how many generators there are, as well as their bidegrees.

As a sequel to \Cref{Intro - AALP theorem}, we pursue a seemingly unrelated situation, namely Rees rings of ideals and modules 
over a \textit{complete intersection ring} $R$. Measures must be taken so that \Cref{Intro - Duality Theorem} can be applied to $\S(E)$, as $R$ is not a polynomial ring. However, doing so reveals how similar this setting is to that of \Cref{Intro - AALP theorem}. A similar observation was made for Rees rings of ideals and modules in \textit{hypersurface rings} in \cite{Weaver23}, relating them to Rees rings of ideals and modules with almost linear presentation \cite{BM16,Costantini21}.

\begin{thmx}\label{Intro - Modules over CIs theorem}
Let $S=k[x_1,\ldots,x_{d+2}]$ with $d\geq 2$ and let $R= S/(f,g)$ where $f,g$ is a $S$-regular sequence of homogeneous polynomials with $\deg f= \deg g=2$. Let $E$ be a torsion-free $R$-module with rank $e$ and projective dimension one, minimally generated by $\mu(E)=d+e$ elements, and presented by matrix $\varphi$. Assume that $\varphi$ consists of linear entries, $E$ satisfies $G_d$, and $I_1(\varphi) = (x_1,\ldots,x_{d+2})R$. The defining ideal $\A$ of $\R(E)$ is generated as follows.
\begin{enumerate}
    \item[(i)] $\A_2$ is minimally generated by one equation of bidegree $(2,d)$.

    \item[(ii)] If $\hgt I(\sigma_1) \geq 3$ then $\A_1$ is minimally generated by $d+2$ equations of bidegree $(1,2d)$. 

    \item[(iii)]  If $\hgt I(\rho_1) \geq 2$ then $\A_0$ is minimally generated by one equation of bidegree $(0,4d)$.
\end{enumerate}
Here $\sigma_1$ and $\rho_1$ are the first maps in the graded strands $\K_1$ and $\K_2$, resolving $\S(E)_1$ and $\S(E)_2$ respectively.
\end{thmx}

As a final application of \Cref{Intro - Duality Theorem}, we consider a specific module in the setting of \Cref{Intro - Modules over CIs theorem}, namely the module of \textit{K\"ahler differentials} $\Omega_{R/k}$, of a two-dimensional complete intersection ring $R$ defined by quadrics. 
The main result of \Cref{Tangent Algebras Section} is as follows.

\begin{thmx}\label{Intro - Module of differentials theorem}
 Let $S=k[x_1,x_2,x_3,x_4]$ for $k$ a field of characteristic zero and let $f,g$ be a homogeneous regular sequence of quadrics such that $R=S/(f,g)$ is normal. Let $\Omega_{R/k}$ denote the module of differentials of $R$ over $k$. The defining ideal $\A$ of $\R(\Omega_{R/k})$ is generated as follows.
 \begin{enumerate}
    \item[(i)] $\A_2$ is minimally generated by one equation of bidegree $(2,2)$.

    \item[(ii)] $\A_1$ is minimally generated by four equations of bidegree $(1,4)$. 

    \item[(iii)] $\A_0$ is minimally generated by one equation of bidegree $(0,n)$, for some positive integer $n$.
\end{enumerate}
\end{thmx}

As mentioned, this is a consequence of \Cref{Intro - Modules over CIs theorem}, with many of the required conditions on the module $\Omega_{R/k}$ being satisfied by conditions on the ring $R$, the normality in particular.

We now briefly describe how this paper is organized. In \Cref{Preliminaries} we introduce much of the preliminary material on Rees algebras of modules required for the scope of this paper. Moreover, we recall the construction of the so-called Buchsbaum-Eisenbud multipliers associated to an acyclic complex. In \Cref{Duality Section} we introduce the duality theorem and many of the tools used for the duration of the paper. As our first application, in \Cref{AALP Section}, we consider Rees algebras of modules with nonlinear presentation, expanding upon ideas in \cite{BM16,Costantini21}. In \Cref{CI Ring Section}, we consider Rees algebras of modules over complete intersection rings, expanding upon the work in \cite{Weaver23}. Finally, in \Cref{Tangent Algebras Section} we apply the main result of \Cref{CI Ring Section} to the defining ideal of the Rees algebra of the module of K\"ahler differentials of a complete intersection ring defined by quadrics. We conclude this paper by posing some open questions and directions for future work.


\section{Preliminaries}\label{Preliminaries}

We now provide the preliminary material required for this paper. We begin by reviewing symmetric algebras and Rees rings of modules. We also recall certain acyclicity criteria and certain constructions associated to an acyclic complex.

\subsection{Rees rings of modules}

In this paper, we consider Rees algebras of ideals and modules, hence we proceed as generally as possible, opting for the latter notion. Let $R$ denote a Noetherian ring and let $E$ be a finitely generated $R$-module of rank $e\geq 1$. Writing $E=Ra_1+\cdots+R a_n$ for a generating set of $E$, there is a natural homogeneous epimorphism of $R$-algebras
$$R[y_1,\ldots,y_n] \longrightarrow \S(E)$$
given by sending $y_i \mapsto a_i \in [\S(E)]_1$, where $\S(E)$ denotes the symmetric algebra of $E$. Writing $\L$ for the kernel of this map, we have an induced isomorphism
$$\S(E) \cong R[y_1,\ldots,y_n] / \L.$$

\begin{defn}The \textit{Rees algebra} of $E$ is the quotient $\R(E) = \S(E)/\tau_R\big(\S(E)\big)$, where $\tau_R\big(\S(E)\big)$ denotes the $R$-torsion submodule of $\S(E)$.  
\end{defn}
With this, the map above induces a second homogeneous epimorphism
$$R[y_1,\ldots,y_n] \longrightarrow \R(E)$$
obtained by further factoring the $R$-torsion of $\S(E)$. Letting $\J$ denote the kernel of this map, there is an isomorphism
$$\R(E) \cong R[y_1,\ldots,y_n] / \J$$
as well. The ideals $\L, \J \subseteq R[y_1,\ldots,y_n]$ are the \textit{defining ideals} of $\S(E)$ and $\R(E)$, respectively.

Whereas the generators of $\J$ are typically enigmatic and difficult to describe, the generators of $\L$ are readily available from a presentation of the module $E$. Indeed, if $R^m \overset{\varphi}{\rightarrow} R^n \rightarrow {E} \rightarrow 0$ is such a presentation, then $\L$ is generated by the linear forms $\ell_1,\ldots,\ell_m$ where 
$$[\ell_1\ldots \ell_m] = [y_1\ldots y_n] \cdot \varphi.$$
From the prior discussion, one has the containment $\L \subseteq \J$. If $\L = \J$, or equivalently $\S(E) \cong \R(E)$, then $E$ is said to be of \textit{linear type}.

Typically one is most interested when $E$ is not of linear type, and so $\S(E)$ and $\R(E)$ differ. To measure this difference, one employs the $\S(E)$-ideal $\A = \J/\L$, as it fits into the short exact sequence
$$0 \longrightarrow \A \longrightarrow \S(E) \longrightarrow \R(E) \longrightarrow 0.$$
 Moreover, as there is an isomorphism $\R(E) \cong \S(E) / \A$, one commonly refers to $\A$ as a defining ideal of $\R(E)$ as well. The ideal $\A$ tends to be simpler than $\J$ and is more easily related to homological constructions, providing more insight.

\begin{defn}
    For $s$ a positive integer, the $R$-module $E$ is said to satisfy $G_s$ if $\mu(E_{\p}) \leq \dim R_{\p} +e -1$ for every $\p \in \spec(R)$ with $1\leq \dim R_{\p} \leq s-1$. The module $E$ is said to satisfies $G_\infty$ if it satisfies $G_s$ for all $s$.
\end{defn}
 Here $\mu(E_{\p})$ denotes the minimal number of generators of $E_\p$ and, as before, $e$ denotes the rank of $E$. It is often convenient to rewrite this condition in terms of Fitting ideals. Recall that for $R^m \overset{\varphi}{\rightarrow} R^n \rightarrow {E} \rightarrow 0$ a presentation of $E$ as before, the $i^{\text{th}}$ \textit{Fitting ideal} is $\fitt_i(E) = I_{n-i}(\varphi)$, the ideal of $(n-i)\times (n-i)$ sized minors of $\varphi$. From \cite[20.6]{Eisenbud}, it follows that $E$ satisfies $G_s$ if $\hgt \fitt_i(E) \geq i-e+2$ for all $e\leq i\leq s+e-2$. The condition $G_s$ typically controls the complexity of the ideal $\A$ and, in many cases, also dictates the prime ideals of $R$ upon which $\A$ is supported.

As mentioned, one is most interested when $E$ is not of linear type, i.e. $\L\neq \J$ or rather $\A\neq 0$. As such, we seek a method to produce nontrivial generators, the Jacobian dual matrix being a lucrative source.

\begin{defn}
Let $R^m\overset{\varphi}{\rightarrow}R^n \rightarrow E\rightarrow 0$ be a presentation of $E$ as above and let $\ell_1,\ldots,\ell_m$ denote the generators of $\L$ once more. There exists an $r\times m$ matrix $\B(\varphi)$ consisting of linear entries in $R[y_1,\ldots,y_n]$ such that
$$[y_1 \ldots y_n] \cdot \varphi=[\ell_1 \ldots \ell_m] = [x_1\ldots x_r]\cdot \B(\varphi) $$
where $(x_1,\ldots,x_r)$ is an ideal containing $I_1(\varphi)$, the ideal of entries of $\varphi$. We say that $\B(\varphi)$ is a \textit{Jacobian dual} matrix of $\varphi$ with respect to the sequence $x_1,\ldots,x_r$.  
\end{defn}

As a consequence of Cramer's rule, one has that $I_r(\B(\varphi)) \subseteq \J$, hence $\B(\varphi)$ is a source of higher-degree equations defining $\R(E)$. Note that $\B(\varphi)$ depends on the sequence $x_1,\ldots,x_r$ and, even for a fixed sequence, $\B(\varphi)$ is not unique in general. However, if $R=k[x_1,\ldots,x_d]$ and the entries of $\varphi$ are linear, then there is a unique Jacobian dual matrix $\B(\varphi)$, with respect to $x_1,\ldots,x_d$ \cite[p.~47]{SUV93}. Moreover, in this case, the entries of $\B(\varphi)$ belong to the subring $k[y_1,\ldots,y_n]$.

\begin{defn}
If $R$ is a local ring with maximal ideal $\m$ and residue field $k$, the \textit{special fiber ring} of $E$ is 
$$\F(E) = \R(E)\otimes_R k \cong \R(E)/\m \R(E).$$ 
Its Krull dimension is called the \textit{analytic spread} of $E$ and is denoted by $\ell(E) = \dim \F(E)$.    
\end{defn}

\subsection{Acyclicity and free resolutions}

One of the themes of this paper is to study certain modules through their resolutions. As such, we briefly recall necessary and sufficient criteria for a complex to be acyclic. Moreover, we also review certain constructions associated to an acyclic complex. For now, assume that $R$ is any commutative ring and

\begin{equation}\label{Complex}
\SelectTips{cm}{}
\xymatrix{C\downdot\,: & 0 \ar[r] & F_n\ar^-{\sigma_n}[r] & \cdots \cdots \ar^-{\sigma_2}[r]& F_1\ar^-{\sigma_1}[r]&F_0}    
\end{equation}
is a complex of free $R$-modules. For each free module write  $f_k = \rk F_k$. Similarly, for each map write $r_k = \rk \sigma_k$, where $\rk \sigma_k = \max\{t\,|\,I_t(\sigma_k)\neq 0\}$. As before, $I_t(\sigma_k)$ denotes the ideal of $t\times t$ minors of $\sigma_k$, after choosing bases and identifying $\sigma_k$ with a matrix. Whereas a matrix representation of $\sigma_k$ does depend on the choice of bases, the ideal of minors $I_t(\sigma_k)$ does not. As a convention, set $I(\sigma_k) = I_{r_k}(\sigma_k)$, the ideal of rank-sized minors.


We now recall the acyclicity criterion of Buchsbaum and Eisenbud \cite{BE73}.

\begin{thm}[{\cite[Cor. 1]{BE73}}]\label{BE-AcyclicityCriteria}
The complex $C\downdot$ in (\ref{Complex}) is acyclic if and only if for $k=1,\ldots,n$, one has
\begin{enumerate}
    \item[(i)] $f_k = r_k +r_{k+1}$ and

    \item[(ii)] either $\grade I(\sigma_k) \geq k$ or $I(\sigma_k) = R$.
\end{enumerate}
\end{thm}

Recall that any map of free $R$-modules $\sigma: F\rightarrow G$ induces a map of exterior powers $\bigwedge^t \sigma: \bigwedge^t F \rightarrow \bigwedge^t G$ for any $t$. Once bases for $F$ and $G$ have been chosen and a corresponding matrix representation of $\sigma$ obtained, the entries of the matrix representation of $\bigwedge^t \sigma$ are precisely the signed $t\times t$ minors of $\varphi$ \cite[A2.3]{Eisenbud}. With this, we recall the notion of Buchsbaum-Eisenbud multipliers \cite{BE74} associated to an acyclic complex.

\begin{thm}[{\cite[3.1]{BE74}}]\label{BEmultipliers}
Suppose that the conditions of \Cref{BE-AcyclicityCriteria} are met and the complex $C\downdot$ in (\ref{Complex}) is acyclic. For $k=1,\ldots,n$, there exists a unique $R$-homomorphism $a_k:R\longrightarrow \bigwedge^{r_k} F_{k-1}$ such that
\begin{itemize}
    \item[(i)] $a_n=\bigwedge^{r_n} \sigma_n \,:\,R= \bigwedge^{r_n}F_n \longrightarrow\bigwedge^{r_n}F_{n-1}$.
    
    \item[(ii)] For all $k<n$, the diagram
    \[
    \SelectTips{cm}{}
    \xymatrix{\bigwedge^{r_k} F_k \ar_-{a_{k+1}^*}[drr] \ar^-{\bigwedge^{r_k}\sigma_k}[rrrr] & & & & \bigwedge^{r_k}F_{k-1}\\
    & & R \ar_-{a_k}[urr]& & }
    \]
    commutes, where $-^*$ denotes the functor $\Hom_R(-,R)$.
    
    \item[(iii)] For all $k>1$, one has $\sqrt{I(a_k)} = \sqrt{I(\sigma_k)}$.
\end{itemize}
\end{thm}

Recall that if $F$ is a free $R$-module of rank $n$, then $\bigwedge^n F$ is a free module of rank one. In other words, there exists an isomorphism $\omega: \bigwedge^n F \overset{\sim}{\rightarrow} R$ called an \textit{orientation} of $F$. With this, it follows that the forms $\bigwedge^i F \otimes \bigwedge^{n-i}F \overset{\wedge}{\longrightarrow} \bigwedge^n F \overset{\omega}{\longrightarrow} R$ induce isomorphisms
$\omega_i : \bigwedge^i F \rightarrow \Hom_R(\bigwedge^{n-i}F,R) = (\bigwedge^{n-i}F)^*$, where $\wedge$ is the usual multiplication in the exterior algebra. 


\section{Bigraded complete intersections and duality}\label{Duality Section}

We now introduce the main tool used throughout the paper, the duality theorem and its application to symmetric algebras and the defining ideal $\A$ of the Rees algebra. As noted, this is not a novel approach to the matter (see e.g. \cite{KM20,KPU Bigraded Structures}), however we aim to extend these techniques to previously untouched settings; see \Cref{CI Ring Section,Tangent Algebras Section}, in particular. Our setting for much of this section is as follows.

\begin{set}\label{Duality Section - Setting}
 Let $R=k[x_1,\ldots,x_d]$, $T=k[y_1,\ldots,y_n]$, and $B= R\otimes_k T$. Endowing $B$ with the bigrading given by $\bideg x_i = (1,0)$ and $\bideg y_i = (0,1)$, let $\fseq = f_1,\ldots,f_d\subseteq \m=(x_1,\ldots,x_d)$ be a $B$-regular sequence of bihomogeneous elements. Write $\bideg f_i =(\alpha_i,\beta_i)$ and let $\S=B/\I$, where $\I=(\fseq)$.
\end{set}

With the assumptions of \Cref{Duality Section - Setting}, we may now state the duality theorem and its consequences. We then describe how it may be applied to Rees algebras of modules and recall the techniques developed in \cite{KM20}.

\subsection{Duality and complete intersections}
With the general setting above established, we now introduce our main tool, the duality theorem. We note that this is only a slight reformulation of the original duality theorem of Jouanolou \cite[Section 3.6]{Jouanolou96}, written in the bigraded setting. We present its statement and give a proof differing from Jouanolou's original arguments. Instead, we provide a proof in the style of Kustin, Polini, and Ulrich \cite[2.4]{KPU Bigraded Structures}, using local cohomology and its interaction with Koszul homology.

\begin{thm}[{\cite[Section 3.6]{Jouanolou96}, \cite[2.4]{KPU Bigraded Structures}}] \label{Duality Theorem}
With the assumptions of \Cref{Duality Section - Setting}, there is a bihomogeneous isomorphism of $B$-modules
$$H_\m^0 (\S) \cong \Hom_T(\S,T)(-\delta,-\tau)$$
where $\delta = (\sum_{i=1}^d \alpha_i) -d$, $\tau =\sum_{i=1}^d \beta_i$, and $H_\m^0 (\S)$ is the zeroth local cohomology module of $\S$ with respect to $\m=(x_1,\ldots,x_d)$. 
\end{thm}

\begin{proof}
 Set $m= \sum_{i=1}^d \alpha_i$ and notice that, since $f_1,\ldots,f_d\subseteq \m$, we have $\alpha_i\geq 1$ for all $i$. Hence $m \geq d$ and so $\delta$ is nonnegative. By the self-duality of the Koszul complex $\K\downdot(\fseq;B)$, there is an isomorphism of complexes of bigraded $B$-modules 
\begin{equation}\label{K self-duality}
 {\rm Hom}_T\big(\K\downdot(\fseq;B),T\big) \cong \K\downdot\big( \fseq;{\rm Hom}_T(B,T)\big)[d] (m,\tau)   
\end{equation}
where $[-]$ denotes homological shift and $(-,-)$ denotes the usual bidegree shift.

Recall that the Koszul complex $\K\downdot(\fseq ;B)$ is acyclic as $\fseq$ is a $B$-regular sequence. As its zeroth homology is $B/\I=\S$, we have the following commutative diagram
\[
\SelectTips{cm}{}
\xymatrix@C=10pt @R=10pt{\K\downdot(\fseq ;B) \,:\, &0\ar[r]& K_d \ar[r] & K_{d-1} \ar[rr]\ar[dr]&  & K_{d-2}\ar[r]& \cdots\cdots\cdots  \ar[r] & K_2\ar[rr]\ar[dr] & & K_1\ar[rr]  \ar[dr]& & K_0 \ar[r] & \S\ar[r] &0\\
& & & & C_{d-1} \ar[ur] \ar[dr]& & & & C_2\ar[ur]\ar[dr] & & C_1\ar[ur]\ar[dr]  \\
& & & 0 \ar[ur] & &0  & & 0\ar[ur]& &  0\ar[ur] & & 0}
\]
where the top row is exact with the $C_i$ its syzygies. With this, we create a series of short exact sequences

\begin{equation}\label{KoszulSES}
\SelectTips{cm}{}
\xymatrix@R=6pt{0\ar[r] & C_1 \ar[r] & K_0 \ar[r] & \S\ar[r] & 0\\
0\ar[r] & C_2 \ar[r] & K_1 \ar[r] & C_1\ar[r] & 0\\
& \vdots & \vdots & \vdots \\
0\ar[r] & C_{d-1} \ar[r] & K_{d-2} \ar[r] & C_{d-2}\ar[r] & 0\\
0\ar[r] & K_d \ar[r] & K_{d-1} \ar[r] & C_{d-1}\ar[r] & 0.}
\end{equation}

Notice that $\grade \m B=d$ and the $K_i$ are $B$-free. Hence by applying the functor $H_\m\updot  (-)$ to the first $d-1$ sequences in (\ref{KoszulSES}) and observing the consequential vanishing, we obtain the isomorphisms
$$H_\m^0(\S) \cong H_\m^1(C_1) \cong H_\m^2(C_2)\cong \cdots\cdots \cong H_\m^{d-1} (C_{d-1}).$$
Now applying $H_\m\updot  (-)$ to the last short exact sequence in (\ref{KoszulSES}), one has
\[
\SelectTips{cm}{}
\xymatrix{0= H_\m^{d-1}(K_{d-1}) \ar[r] & H_\m^{d-1}(C_{d-1}) \ar[r] & H_\m^d(K_d) \ar^-{\partial}[r] & H_\m^d(K_{d-1})}
\]
where $\partial$ is precisely the $d^{\text{th}}$ differential in $\K\downdot\big(\fseq ;H_\m^d(B)\big)$, as $K_d$ and $K_{d-1}$ are $B$-free and local cohomology commutes with direct sums. Hence
$$H_\m^0(\S) \cong H_\m^{d-1}(C_{d-1}) \cong H_d\big(\K\downdot\big(\fseq;H_\m^d(B)\big)\big).$$
However, by \Cref{H(B) Lemma} and the isomorphism in (\ref{K self-duality}) we have  
\begin{align*}
    H_d\big(\K\downdot\big(\fseq;H_\m^d(B)\big)\big) & \cong H_d\big(\K\downdot\big(\fseq;{\rm Hom}_T(B,T)(d,0)      \big)\big)\\[1ex] 
    & \cong H_0\big({\rm Hom}_T\big(\K\downdot(\fseq;B),T\big)\big) (-m,-\tau)(d,0)\\[1ex]
    &\cong {\rm Hom}_T(\S,T) (-m+d,-\tau)
\end{align*}
and the claim follows since $-m+d=-\delta$.
\end{proof}

\begin{lemma}\label{H(B) Lemma}
There is a bihomogeneous isomorphism $H_\m^d(B) \cong \Hom_T(B,T)(d,0)$.
\end{lemma}

\begin{proof}
    Recall that $B = R\otimes_k T$ where $R=k[x_1,\ldots,x_d]$ and $T= k[y_1,\ldots,y_n]$. With this and noting that $B$ is $R$-flat, we have
    $$H_\m^d(B) \cong  H_\m^d(R\otimes_R B) \cong 
    H_\m^d(R) \otimes_R B  \cong  H_\m^d(R) \otimes_k T.$$
Moreover, as $\m$ is the homogeneous maximal ideal of $R$, by Serre duality \cite[III.7.1]{Hartshorne}, we have a graded isomorphism of $R$-modules $H_\m^d(R) \cong \Hom_k(R,k)(d)$. Hence by flatness once more, and adjusting to the bigrading of $B$, we have
$$H_\m^d(R) \otimes_k T \cong \Hom_k(R,k)(d) \otimes_k T \cong \Hom_T(B,T)(d,0)$$
and the claim follows.
\end{proof}

Although the rings and modules here are bigraded, it is often convenient to fix one component of the bigrading to obtain singly graded modules. We introduce notation to permit this.

\begin{notat}\label{T-Grading}
    For any bigraded $B$-module $M = \bigoplus_{i,j} M_{(i,j)}$, write $M_i$ to denote the graded $T$-module 
    $$M_i= M_{(i,*)} = \bigoplus_j M_{(i,j)}.$$
\end{notat}

With this convention, we obtain the following consequence of \Cref{Duality Theorem}.

\begin{cor}\label{Duality Corollary}
 With the assumptions of \Cref{Duality Section - Setting}, write $A = H_\m^0(\S)$. 
\begin{enumerate}
    \item[(a)] There is a graded isomorphism of $T$-modules $A_i \cong \Hom_T(\S_{\delta-i},T(-\tau))$
for all $i$.

\item[(b)] For all $i<0$ and $i> \delta$, we have $A_i = 0$.

\item[(c)] There is a graded isomorphism $A_\delta \cong T(-\tau)$.
\end{enumerate}
\end{cor}

\begin{proof}
Assertion (a) follows immediately from \Cref{Duality Theorem}. Part (b) then follows from (a) noting that $\S$ is nonnegatively graded as a $T$-module. Lastly, assertion (c) follows from (a) as $\S_0 \cong T$.
\end{proof}

Notice that by \Cref{Duality Corollary}, $A_\delta$ is a cyclic $T$-module generated by an equation of bidegree $(\delta, \tau)$. Moreover, we can actually produce such a generator.

\begin{prop}\label{A-delta generator}
The $T$-module $A_\delta$ is generated as $A_\delta = \langle \det \B\rangle$ where $\B$ is a $d\times d$ matrix $\B$ consisting of bihomogeneous entries in $B$ such that $[f_1\ldots f_d ] =[x_1 \ldots x_d]\cdot \B$.     
\end{prop}

\begin{proof} 
Writing $\I=(f_1,\ldots,f_d)$ and $\bideg f_i = (\alpha_i,\beta_i)$ as before, recall that $\alpha_i \geq 1$ for all $i$. Notice that the entries in column $i$ of $\B$ have bidegree $(\alpha_i-1,\beta_i)$, hence it follows that $\det \B$ is a form of bidegree $(\delta ,\tau)$ so long as $\det \B \neq 0$. Moreover, the assertion will follow once $\det \B$ has been shown to be nonzero modulo $\I$. 

We note that the matrix $\B$ is not unique in general, however the ideal $\I +(\det\B)$ is unique. Indeed, since $\I$ and $(x_1,\ldots,x_d)$ are generated by regular sequences, it is well-known (see e.g. \cite[2.3.10]{BH93}) that
$$\I:(x_1,\ldots,x_d)= \I + (\det \B).$$
Notice that $\I:(x_1,\ldots,x_d) \neq \I$, as $A$ is nonzero by \Cref{Duality Corollary}(c). Hence it follows that $\det \B \notin \I$, and so $\det \B$ is nonzero modulo $\I$.
\end{proof}

As noted in \Cref{Preliminaries}, such a transition matrix $\B$ is called a Jacobian dual when $f_1,\ldots,f_d$ are the linear equations defining a symmetric algebra. This will be the case in \Cref{AALP Section}, however in \Cref{CI Ring Section} we will take such a matrix $\B$ to be a \textit{modified} Jacobian dual, as introduced in \cite{Weaver23}.

\subsection{Certain complexes}

With \Cref{Duality Corollary}, it suffices to study the $T$-duals of graded components of the complete intersection $\S$, in order to understand the $T$-graded components of $A$. Moreover, information regarding the module $\Hom_T(\S_{\delta-i},T)$ can be obtained from a free resolution of $\S_{\delta-i}$. As $\S$ is a complete intersection ring, it is resolved by the Koszul complex. Hence we may take its $T$-graded strands to resolve the graded components of $\S$.

\begin{rem}\label{Koszul}
With the assumptions of \Cref{Duality Section - Setting}, recall that the Koszul complex $\K\downdot(\fseq;B)$ is a $B$-resolution of $\S = B/(\fseq)$. Moreover, as $\bideg f_i = (\alpha_i, \beta_i)$, the bigraded Koszul complex is
\begin{equation}\label{KoszulComplex}
K\downdot(\fseq;B)\,:\, 0 \longrightarrow B (-\textstyle\sum \alpha_i , - \textstyle\sum \beta_i)\longrightarrow \cdots\cdots \longrightarrow \overset{d}{\bigoplus} B(-\alpha_i, -\beta_i)\longrightarrow B.
\end{equation}
With this, we may obtain graded $T$-resolutions for each $\S_t$. Indeed, writing $\K_t$ to denote the graded strand of (\ref{KoszulComplex}) in degree $t$, the graded free resolution of $\S_t$ is
\begin{equation}\label{KoszulStrand}
    \K_t \, : \, 0\longrightarrow F_m \overset{\sigma_m}{\longrightarrow} F_{m-1} \overset{\sigma_{m-1}}{\longrightarrow} \cdots\cdots \overset{\sigma_1}{\longrightarrow} F_0
\end{equation}
for some $m\leq d$, where
\begin{equation}\label{KoszulStrandModule}
F_i = \bigoplus_{1\leq j_1\leq \cdots\leq j_i\leq d} T(-(\beta_{j_1}+\cdots+\beta_{j_i}))^{\binom{t-(\alpha_{j_1} + \cdots+\alpha_{j_i}) +d-1}{d-1}}
\end{equation}
for $0\leq i\leq m$, noting that $\binom{n}{k}=0$ when $0\leq n <k$.
\end{rem}

With the resolution of $\S_t$ in (\ref{KoszulStrand}), notice that by \Cref{Duality Corollary} we have $A_{\delta-t} \cong \ker \sigma_1^* (-\tau)$. 
With this, we recall the complex of Kim and Mukundan from \cite{KM20}. We present its statement slightly differently than as originally written, noting that the construction of this complex shows that it may be applied in the setting of \textit{any} bigraded complete intersection ring, as in \Cref{Duality Section - Setting}.

\begin{thm}[{\cite[Theorems 13-16]{KM20}}]\label{KMcomplex}
With the assumptions of \Cref{Duality Section - Setting} and $\K_t$ the $T$-resolution of $\S_t$ as in \Cref{Koszul}, there exists a complex of graded $T$-modules
\[
\SelectTips{cm}{}
\xymatrix{\bigwedge^{f_0-r_1-1} F_0 \otimes T(-s) \ar^-{\omega\circ {\wedge} \circ (\id \otimes a_1)}[rr] & & F_0^* \ar^{\sigma_1^*}[r] & F_1^*.}
\]
Here $\omega : \bigwedge^{f_0-1} F_0 \overset{\sim}{\rightarrow} F_0^*$ is the isomorphism induced by an orientation on $F_0$, $\wedge$ is multiplication within the exterior algebra, and $a_1: T(-s) \rightarrow \bigwedge^{r_1} F_0$ is the map from \Cref{BEmultipliers} with degree shift $s$ making it a homogeneous map of free $T$-modules. Moreover, we have the following.
\begin{enumerate}

    \item[(a)] If $\rk \S_t =1$, then this complex is exact if and only if $\grade I(a_1) \geq 2$.

    \item[(b)] Assume $\rk \S_t =2$ and consider the following statements:
    
\begin{enumerate}

    \item[(i)] $\S_t$ satisfies Serre's condition $S_2$.
    
    \item[(ii)] $\grade I(\sigma_t) \geq t+2$ for $t=1,\ldots, m$.
    
    \item[(iii)] The complex is exact.
\end{enumerate}
Then $(i)$ implies $(ii)$ which implies $(iii)$.
\end{enumerate}
\end{thm}

 In particular, if the complex in \Cref{KMcomplex} is exact, then $A_{\delta -t}$ is generated by $\binom{f_0}{f_0-r_1-1}$ elements of bidegree $(\delta -t,s+\tau)$. Note that, by \Cref{BEmultipliers}, it also suffices to show that $\hgt I(\sigma_1) \geq 2$ in \Cref{KMcomplex}(a), which is often more convenient.

\begin{rem}
    The degree shift $s$ for $a_1: T(-s) \rightarrow \bigwedge^{r_1} F_0$ in \Cref{KMcomplex} may be computed by adapting \Cref{BEmultipliers} to graded free modules. However, as the construction of these maps is recursive, this value must be computed iteratively. While it is possible to produce a formula for $s$ with the assumptions of \Cref{Duality Section - Setting} using the data in (\ref{KoszulStrand}) and (\ref{KoszulStrandModule}), we forgo this option as it will be impractical for our purposes. Indeed, the complexes in the proceeding sections will be remarkably short, hence this value will be easily computed. We refer the curious reader to \cite[Rem. 12]{KM20} for the general procedure and a formula for $s$ when $\beta_1=\cdots =\beta_d = 1$ in \Cref{Duality Section - Setting}.
\end{rem}

\subsection{Complete intersections and symmetric algebras}

We intend to use \Cref{Duality Corollary} and \Cref{KMcomplex} when $\S$ is the symmetric algebra $\S(E)$ of a module $E$ and $A$ is the defining ideal $\A$ of the Rees ring $\R(E)$, as a quotient of $\S(E)$. However, we must then ensure that $\S(E)$ is a complete intersection ring and that $\A$ coincides with the local cohomology module $H_\m^0(\S(E))$. We provide sufficient criteria for both phenomena to occur.

\begin{prop}\label{S(E) Complete Intersection Criteria}
    Let $R$ be a Cohen-Macaulay local ring of dimension $d$ with maximal ideal $\m$ and let $E$ be a torsion-free $R$-module with rank $e\geq 1$. Assume that $E$ satisfies $G_d$, $\mu(E) = d+e$, and $E$ has projective dimension one. 

    \begin{enumerate}
        \item[(a)] Writing $B = R[y_1,\ldots,y_{d+e}]$ and $\L$ the ideal such that $\S(E) \cong B/\L$, the ideal $\L$ is generated by a $B$-regular sequence. 

        \item[(b)] Write $\J$ and $\A$ for the ideals such that $\R(E) \cong B/\J \cong \S(E)/\A$. The module $E$ is of linear type on the punctured spectrum of $R$, i.e. $\A_\p = 0$ for all $\p \in \spec(R) \setminus\{\m\}$. In particular, $\A \cong H_\m^0(\S(E))$ and $\J = \L :\m^\infty$.
    \end{enumerate}
\end{prop}

\begin{proof}
Since $E$ has projective dimension 1 and $\mu(E) =d+e$, it has a minimal free resolution of the form
\begin{equation}\label{Presentation}
    0 \longrightarrow R^d \overset{\varphi}{\longrightarrow} R^{d+e} \longrightarrow {E} \longrightarrow 0.
\end{equation}
Hence $\L$ is generated by the $d$ linear forms $\L=(\ell_1,\ldots,\ell_d)$ where $[\ell_1\ldots,\ell_d] = [y_1\ldots y_{d+e}]\cdot \varphi$. Thus it suffices to show that $\dim \S(E) =d+e$ to conclude that $\hgt \L =d$, and so $\ell_1,\ldots,\ell_d$ is a $B$-regular sequence. By the Huneke-Rossi formula \cite{HR86}, we have 
$$\dim \S(E) = \sup\{\mu(E_\p) + \dim R/\p \,|\, \p\in \spec(R)\}$$
and we claim that the maximum value is $d+e$. If $\p \neq \m$, from the condition $G_d$ it follows that
$$\mu(E_\p)+\dim R/\p \leq \hgt \p+e -1 +\dim R/ \p = d+e-1.$$
Now taking $\p=\m$, we have $\mu(E_\m) + \dim R/\m = \mu(E)+0 = d+e$. Hence $\dim \S(E) = d+e$ which shows (a).

To prove (b), it is well known that one has the containment $H_\m^0(\S(E)) \subseteq \A$. Notice that since $E$ satisfies $G_d$, each localization $E_\p$ satisfies $G_\infty$ for any non-maximal prime $R$-ideal $\p$. Hence by \cite[Prop. 3 and 4]{Avramov81} it follows that $E_\p$ is of linear type for any such prime, and so $\A_\p =0$ for all $\p \neq \m$. Hence $\A$ is annihilated by some power of $\m$ and so $\A \subseteq H_\m^0(\S(E))$ as well. As $\A = \J/\L$, this also shows that $\J$ is the claimed saturation.
\end{proof}

We remark that, although it is stated in the local setting, \Cref{S(E) Complete Intersection Criteria} also applies to a graded module $E$ over $R=k[x_1,\ldots,x_d]$ with $\m=(x_1,\ldots,x_d)$ the unique homogeneous maximal ideal, as in \Cref{Duality Section - Setting}. We also note that, although the condition $\mu(E) =d+e$ might seem restrictive, if $\mu(E) < d+e$ with the remaining assumptions, then $E$ satisfies $G_\infty$ and is hence of linear type \cite[Prop. 3 and 4]{Avramov81}. Thus this is the first nontrivial setting, in regard to the number of generators $\mu(E)$.

\begin{rem}\label{Ideal Remark}
In the setting of \Cref{S(E) Complete Intersection Criteria}, if $e=1$ then $E$ is isomorphic to a perfect ideal of grade two. Indeed, as $E$ is torsion-free with rank $e=1$, it is isomorphic to an $R$-ideal $I$. Since $E\cong I$ has projective dimension 1, it follows from (\ref{Presentation}) that it is presented by a $(d+1)\times d$ matrix. Moreover, the condition $G_d$ implies that $\hgt I_d(\varphi) = \hgt \fitt_{1}(I) \geq 2$ and the claim follows from the Hilbert-Burch theorem \cite[20.15]{Eisenbud}. 
\end{rem}



\section{Rees algebras of modules with nonlinear presentation}\label{AALP Section}

In this section, we study Rees algebras of modules with particular restrictions on their presentations. If $R=k[x_1,\ldots,x_d]$ and $E$ is a graded $R$-module homogeneously generated in a single degree, notice that $E$ has a presentation matrix with entries concentrated in a single degree within each column. As these degrees dictate the bidegrees of the equations defining the symmetric algebra $\S(E)$, one attempts to study the Rees ring $\R(E)$ with certain restrictions on them. As such, we recall the notion of the \textit{type} of a matrix from \cite{KM20}. A $m\times n$ matrix is said to be of type $(d_1,\ldots,d_n)$ if it consists of entries of degree $d_i$ in column $i$ for $1\leq i\leq n$.

As noted in the introduction, there has been an extensive study of Rees rings of ideals and modules with linear presentation, 
 i.e. those with presentation matrices of type $(1,\ldots,1)$. Moreover, there has also been success for ideals and modules with \textit{almost} linear presentation, 
 namely those with presentation matrices of type $(1,\ldots,1,m)$ for some $m\geq 1$. To continue this study, in this section we consider modules with presentation of type $(1,\ldots,1,2,2)$ and their Rees rings.

\begin{set}\label{AALP Setting}
Let $R=k[x_1,\ldots,x_d]$ with $d\geq 3$ and let $E$ be a torsion-free $R$-module with rank $e$ and projective dimension one, minimally generated by $\mu(E)=d+e$ elements. 
Letting
$$0 \longrightarrow R^d \overset{\varphi}{\longrightarrow} R^{d+e} \longrightarrow {E} \longrightarrow 0$$
denote a minimal free resolution of $E$, assume that $\varphi$ is of type $(1,\ldots,1,2,2)$, i.e. $\varphi = [\varphi'\,|\,\varphi'']$ where $\varphi'$ has $d-2$ columns with linear entries and $\varphi''$ consists of two columns of quadrics. Further assume that $I_1(\varphi) = (x_1,\ldots,x_d)$ and $E$ satisfies the condition $G_d$.
\end{set}

As noted, one reduces to the study of Rees rings of a notable class of ideals when the rank of $E$ is $e=1$.

\begin{rem}\label{AALP Ideal Setting}
Following \Cref{Ideal Remark}, if $e=1$ then $E$ is isomorphic to an $R$-ideal $I$, which is perfect of grade two. Moreover, from the type of $\varphi$ and the Hilbert-Burch theorem \cite[20.15]{Eisenbud}, it follows that $I$ is minimally generated by $\mu(I)=d+1$ homogeneous forms of degree $d+2$.
\end{rem}

Regardless, we consider the more general situation in \Cref{AALP Setting} throughout. 
We briefly recall the notation required to begin our treatment of the symmetric algebra $\S(E)$.

\begin{notat}\label{AALP Notation}
With the assumptions of \Cref{AALP Setting}, write $T=k[y_1,\ldots,y_{d+e}]$ and let $B=R\otimes_k T \cong R[y_1,\ldots,y_{d+e}]$ with the bigrading given by $\bideg x_i=(1,0)$ and $\bideg y_i=(0,1)$. Recall that the symmetric algebra is $\S(E) \cong B/\L$ where $[\ell_1\ldots\ell_d] = [y_1\ldots y_{d+e}]\cdot \varphi$ and $\L=(\ell_1,\ldots,\ell_d)$. Write $\A$ and $\J$ to denote the ideals defining the Rees algebra $\R(E)$ as quotients of $\S(E)$ and $B$ respectively, as in \Cref{Preliminaries}. 
\end{notat}

By \Cref{S(E) Complete Intersection Criteria}, $\ell_1,\ldots,\ell_d$ is a $B$-regular sequence and $\A = H_\m^0(\S(E))$ where $\m=(x_1,\ldots,x_d)$, hence we may apply \Cref{Duality Corollary} and \Cref{KMcomplex}. Moreover, with the conditions on $\varphi$, the bidegrees of $\ell_1,\ldots \ell_{d-2}$ (recall $d\geq 3$) are $(1,1)$, while the bidegrees of $\ell_{d-1}$ and $\ell_d$ are $(2,1)$. With this, and adopting the grading scheme in \Cref{T-Grading} throughout, we may produce $T$-resolutions of the graded components of $\S(E)$ from the bigraded Koszul complex of $\ell_1,\ldots,\ell_d$ following \Cref{Koszul}.

\begin{prop}\label{AALP Section Ranks and Complexes}
    With $R$ and $E$ as in \Cref{AALP Setting} and $\delta$ and $\tau$ as in \Cref{Duality Theorem}, we have $\delta =2$, $\tau = d$, and $\S(E)_0$ is free of rank one. Moreover, $\rk \S(E)_1=2$ and $\rk \S(E)_2=1$.
\end{prop}

\begin{proof}
The first statement is clear. As for the second assertion, we construct the graded $T$-resolutions of $\S(E)_1$ and $\S(E)_2$ following \Cref{Koszul}. Their ranks are then computed using \Cref{BE-AcyclicityCriteria}. 

\begin{enumerate}[itemsep=0.5mm]
    \item[(a)] For the resolution of $\S(E)_1$, we have 
\begin{equation}\label{aalpS1sequence}
  \SelectTips{cm}{}
\xymatrix{0 \ar[r]&\overset{d-2}{\bigoplus} T(-1)^{\binom{d-1}{d-1}} \ar^-{\sigma}[r] & T^{\binom{d}{d-1}} \ar[r] & \S(E)_1 \ar[r]& 0.}  
\end{equation}

Thus $\rk \S(E)_1= d- (d-2) = 2$.

\item[(b)] For the resolution of $\S(E)_2$, we must consider two cases, as the projective dimension of $S(E)_2$ depends on $d$.

\vspace{0.5mm}

\begin{enumerate}[itemsep=0.5mm]
\item[(i)] If $d=3$, the graded $T$-resolution of $\S(E)_2$ is 
\begin{equation}\label{aalpS2sequenced=3}
\SelectTips{cm}{}
\xymatrix{0 \ar[r] &  T(-1)^{\binom{3}{2}} \oplus \overset{2}{\bigoplus} T(-1)^{\binom{2}{2}} \ar^-{\rho}[r] & T^{\binom{4}{2}} \ar[r] & \S(E)_2 \ar[r]& 0}
\end{equation}
and so $\rk \S(E)_2 = 6- 3-2 =1$.

\item[(ii)] If $d\geq 4$. The graded $T$-resolution of $\S(E)_2$ is 
\begin{equation}\label{aalpS2sequencedgeq4}
\SelectTips{cm}{}
\xymatrix@C=16pt{0 \ar[r]&\overset{\binom{d-2}{2}}{\bigoplus} T(-2)^{\binom{d-1}{d-1}} \ar^-{\gamma}[r] & \overset{d-2}{\bigoplus} T(-1)^{\binom{d}{d-1}} \oplus \overset{2}{\bigoplus} T(-1)^{\binom{d-1}{d-1}} \ar^-{\rho}[r] & T^{\binom{d+1}{d-1}} \ar[r] & \S(E)_2 \ar[r]& 0}
\end{equation}
and so $\rk \S(E)_2 =\binom{d+1}{d-1}-d(d-2)-2+\binom{d-2}{2} =1$.
\end{enumerate}
\vspace{0.5mm}
Thus in either case, we have $\rk S(E)_2 =1$.\qedhere
\end{enumerate}
\end{proof}

Recall that (\ref{aalpS1sequence}) is the graded strand in degree 1 of the bigraded Koszul complex of $\ell_1,\ldots, \ell_d$. Taking the standard monomial bases, ordered lexicographically, for the free $T$-modules involved, $\sigma$ may be realized as the $d\times (d-2)$ matrix with linear entries of $T$, such that $[\ell_1,\ldots,\ell_{d-2}] = [x_1\ldots x_d]\cdot \sigma$. In other words, $\sigma$ is precisely  the Jacobian dual of $\varphi'$, $\sigma = \B(\varphi')$, where $\varphi'$ is the submatrix of $\varphi$ as in \Cref{AALP Setting}.

\begin{rem}\label{AALP Section I(sigma)}
With the map $\sigma$ in (\ref{aalpS1sequence}), we have $\hgt I(\sigma) \geq 2$.
\end{rem}

\begin{proof} 
Recall that $\sigma$ may be taken as the transition matrix such that $[\ell_1,\ldots,\ell_{d-2}] = [x_1\ldots x_d]\cdot \sigma$. As $\sigma$ consists of entries in $T$, it follows that $(\ell_1,\ldots,\ell_{d-2})$ is the ideal defining the symmetric algebra of the $T$-module $\coker \sigma =\S(E)_1$. Thus by \cite[6.6, 6.8]{HSV83} it suffices to show that this ring is a domain.


Notice that as $\sigma$ consists of entries in $T$ and $\varphi'$ consists of entries in $R$, there is an isomorphism of symmetric algebras $\S_T(\S(E)_1)\cong \S_R(M)$, where $M=\coker \varphi'$. Moreover, it can easily be seen that since $E$ has projective dimension 1 and satisfies $G_d$, $M$ does as well. However, note that $\rk M=e+2$ and $\mu(M) = d+e <d+(e+2)$, hence $M$ satisfies $G_\infty$ and is thus of linear type \cite[Prop. 3 and 4]{Avramov81}. Thus $\S_R(M)$ is a domain and so $\S_T(\S(E)_1)$ is as well.
\end{proof}

We note that $\hgt I(\sigma)\leq 3$ by \cite[Thm. 1]{EN62} and if this maximum codimension is attained, then the complex in \Cref{KMcomplex} is exact. Whereas we cannot ensure this maximum height is achieved, we remark that this is only a \textit{sufficient} condition for exactness of this complex.

\begin{thm}\label{AALP Section Main Result}
With $E$ as in \Cref{AALP Setting} and $\A$ the ideal defining $\R(E)$ as a quotient of $\S(E)$, we have the following.
\begin{enumerate}
    \item[(a)] $\A_2$ is generated as $\A_2 = \langle\det \B(\varphi)\rangle$ where $\B(\varphi)$ is a Jacobian dual of $\varphi$, with respect to $x_1,\ldots,x_d$. Moreover, this is an equation of bidegree $(2,d)$.

    \item[(b)] If $\hgt I(\sigma) = 3$, then $\A_1$ is minimally generated by $d$ equations of bidegree $(1,2d-2)$. 

    \item[(c)] $\A_0$ is minimally generated by one equation. If $\hgt I(\rho) \geq 2$, then this generator has bidegree $(0,4d-4)$.
\end{enumerate}
\end{thm}

\begin{proof}
Part (a) follows from \Cref{Duality Corollary} and \Cref{A-delta generator}. For the first part of (c), notice that $\A_0 \subseteq \S(E)_0 = T$, and so $\A_0$ is a $T$-ideal. Recall that the special fiber ring of $E$ is $\F(E)\cong \R(E)/\m \R(E)$ and it follows that $\A_0$ defines $\F(E)$ as a quotient of $T=k[y_1,\ldots,y_{d+e}]$, i.e. $\F(E) \cong T/\A_0$. Moreover, it is well known that $E$ has maximal analytic spread, i.e. $\ell(E) = d+e-1$, in this setting, which may be seen from \cite[3.10]{SUV03} and \cite[4.3]{UV93}. As this is the dimension of $\F(E)$, a domain, it follows that $\A_0$ is a principal $T$-ideal.

Part (b) and the second assertion of (c) follow from \Cref{KMcomplex} and \Cref{AALP Section Ranks and Complexes}, once the degree shift of the map $a_1$ has been computed for resolutions (\ref{aalpS1sequence}) -- (\ref{aalpS2sequencedgeq4}).

\begin{enumerate}[itemsep=0.5mm]
    \item[(a)] For sequence (\ref{aalpS1sequence}), notice that $a_1 = \bigwedge^{d-2} \sigma$, following  \Cref{BEmultipliers}. Hence
    $$a_1\,:\,\bigwedge^{d-2} \big(T(-1)^{d-2}\big ) \cong T(-(d-2)) \longrightarrow  \bigwedge^{d-2} (T^d)$$
    and so $s = d-2$. 

    \item[(b)] For the $T$-resolution of $S(E)_2$, we must consider the two cases from \Cref{AALP Section Ranks and Complexes}. 

    \vspace{0.5mm}
    \begin{enumerate}[itemsep=0.5mm]
        \item[(i)] If $d=3$, consider the $T$-resolution of $\S(E)_2$ in (\ref{aalpS2sequenced=3}). By \Cref{BEmultipliers} we have $a_1 = \bigwedge^5 \rho$, hence
        $$a_1\,:\, \bigwedge^5 \big(T(-1)^5)  \cong T(-5) \longrightarrow \bigwedge^5 \big(T^6\big)$$
        and so the degree shift of $a_1$ is $s = 5$. 

        \item[(ii)] If $d\geq 4$, consider the $T$-resolution of $\S(E)_2$ in (\ref{aalpS2sequencedgeq4}). As this is a resolution of length two, we begin with $a_2$ as the construction of the maps in \Cref{BEmultipliers} is recursive. Note that 
        $$a_2 = \bigwedge^{\binom{d-2}{2}} \gamma\,:\, \bigwedge^{\binom{d-2}{2}} T(-2)^{\binom{d-2}{2}} \longrightarrow \bigwedge^{\binom{d-2}{2}}T(-1)^{d(d-2)+2}$$
        and so the entries in a matrix representation of $a_2$ have degree $\binom{d-2}{2}$. By \Cref{BEmultipliers} we have the commutative diagram
        \[
    \SelectTips{cm}{}
    \xymatrix{\bigwedge^{r} T(-1)^{d(d-2)+2} \ar_-{a_{2}^*}[drr] \ar^-{\bigwedge^{r}\rho}[rrrr] & & & & \bigwedge^{r} T^{\binom{d+1}{d-1}}\\
    & & T(-s) \ar_-{a_1}[urr]& & }
    \]
    where $r = \rk \rho = d(d-2)+2 - 
    \binom{d-2}{2}$ by \Cref{BE-AcyclicityCriteria}. Notice that the degrees of the entries in a matrix representation of $\bigwedge^{r}\rho$ are all $r$ as $\rho$ consists of linear entries. Hence
    $$s= r - \binom{d-2}{2} = d(d-2)+2 - (d-2)(d-3) = 3d-4.$$
        \end{enumerate}
        Thus $s=3d-4$ in either case.
\end{enumerate}
The claim now follows from \Cref{KMcomplex}, noting that $\tau=d$ by \Cref{AALP Section Ranks and Complexes}.
\end{proof}

\begin{rem}\label{AALP Section Ideal Remark}
We note that the shapes of the $T$-resolutions (\ref{aalpS1sequence}) -- (\ref{aalpS2sequencedgeq4}) of the graded components of $\S(E)$ are independent of the rank $e$ of the module $E$. Hence if $e=1$, then one achieves the same result for the defining ideal of the Rees algebra of an ideal $I$ as in \Cref{AALP Ideal Setting}.   
\end{rem}




\begin{rem}\label{AALP expected A}
Notice that if both $\hgt I(\sigma) = 3$ and $\hgt I(\rho) \geq 2$ in \Cref{AALP Section Main Result}, then it is understood how the defining ideals $\A$ and $\J$ of $\R(E)$ are generated. For instance, the ideal $\J$ is then minimally generated as
\begin{center}
 \begin{tabular}{|c |c |c|} 
 \hline
 & Bidegree & Number of generators \\ [0.5ex] 
 \hline
 $\ell_1,\ldots,\ell_{d-2}$ & $(1,1)$ & $d-2$ \\ 
 \hline
  $\ell_{d-1},\ell_d$ & $(2,1)$ & $2$ \\
 \hline
 $\A_0$ & $(0,4d-4)$ &  $1$\\
 \hline
 $\A_1$ & $(1,2d-2)$& $d$ \\
 \hline
 $\A_2 = \langle\det \B(\varphi)\rangle$ & $(2,d)$ & $1$ \\
 \hline
\end{tabular}
\end{center} 
In particular, $\J$ is minimally generated by $2d+2$ elements and $\A$ is minimally generated by $d+2$ elements. 
\end{rem}

Unfortunately, it is not necessarily the case that both conditions of \Cref{AALP Section Main Result} are satisfied. Nevertheless, we consider the number of generators above and their bidegrees to be the \textit{predicted} forms of $\J$ and $\A$. Indeed, the conditions of \Cref{AALP Section Main Result} are frequently satisfied, and it is likely that the behavior in \Cref{AALP expected A} occurs if the entries of $\varphi$ are \textit{sufficiently general}, similar to the main result of \cite{KM20}. We present an example, for Rees rings of ideals, where ideals of minors of the maps in (\ref{aalpS1sequence}) -- (\ref{aalpS2sequencedgeq4}) fail to have the expected codimension.

\begin{ex}\label{AALP Example - I(sigma) Height 2}
Let $R=\mathbb{Q}[x_1,x_2,x_3,x_4]$, let
    $$\varphi=\begin{bmatrix}
     x_1&0&x_4^2&x_2x_3\\
     x_2&0&x_1x_2&x_3^2\\
     0&0&x_1^2&x_2^2\\
     0&x_3&x_2^2&x_2x_3\\
     0&x_4&x_4^2&x_2^2
    \end{bmatrix}$$
    be a matrix as in \Cref{AALP Setting}, and consider the ideal $I=I_4(\varphi)$. Computations through \textit{Macaulay2} \cite{Macaulay2} show that $\hgt I =2$, hence $I$ is perfect of grade 2 \cite[20.15]{Eisenbud}, and $I$ satisfies the condition $G_4$. Thus the assumptions of \Cref{AALP Ideal Setting} and \Cref{AALP Setting} are met. However, the matrix $\sigma=\B(\varphi')$ in (\ref{aalpS1sequence}) is
$$\sigma = \begin{bmatrix}
    y_1&0\\
    y_2&0\\
    0&y_4\\
    0&y_5
\end{bmatrix}
$$
and it can be seen that $\hgt I(\sigma) = 2$, hence the conditions of \Cref{AALP Section Main Result} are not met. Moreover, computations through \textit{Macaulay2} \cite{Macaulay2} show that $\A_1$ is minimally generated by two equations of bidegree $(1,5)$, differing from the predicted behavior in \Cref{AALP expected A}. However, the map $\rho$ in (\ref{aalpS2sequencedgeq4}) has $\hgt I(\rho) = 2$, and so $\A_0$ is generated by one equation of bidegree $(0,12)$, following \Cref{AALP Section Ideal Remark} and \Cref{AALP Section Main Result}.
\end{ex}

\begin{quest}
    Noting that $\A_2$ is generated by the determinant of a Jacobian dual matrix, it is curious if the remaining equations of $\A$ can be described by related constructions. This was shown to be the case when the presentation $\varphi$ is almost linear, noted for perfect ideals of grade two ($e=1$) in \cite{BM16,CHW08} and more generally for modules of projective dimension one in \cite{Costantini21}.
\end{quest}


\section{Rees algebras of modules over complete intersection rings}\label{CI Ring Section}

In this section, we study Rees rings of $R$-modules when $R$ is a complete intersection ring. As noted in the introduction, there is much geometric motivation to study Rees rings when the ground ring $R$ is not a polynomial ring. We note that there has been recent success in the study of Rees algebras of ideals and modules over \textit{hypersurface rings} \cite{Weaver23,Weaver24}. Expanding upon this, we consider modules over complete intersection rings and their Rees algebras. Although this seems like a drastically different setting than that of the previous section, we proceed in a similar manner.

\begin{set}\label{CI Ring Setting}
Let $S=k[x_1,\ldots,x_{d+2}]$ and let $R= S/(f,g)$ where $f,g$ is a $S$-regular sequence of homogeneous polynomials with $\deg f= \deg g=2$. Let $E$ be a torsion-free $R$-module with rank $e$ and projective dimension one, minimally generated by $\mu(E)=d+e$ elements. Let 
$$0 \longrightarrow R^d \overset{\varphi}{\longrightarrow} R^{d+e} \longrightarrow {E} \longrightarrow 0$$
be a minimal free resolution of $E$ and assume that $\varphi$ consists of linear entries in $R$. Assume that $E$ satisfies $G_d$ and $I_1(\varphi) = \overline{(x_1,\ldots,x_{d+2})}$, where $\overline{\,\cdot\,}$ denotes images modulo $(f,g)$.
\end{set}

As before, we recover the case of Rees algebras of perfect ideals of grade two. 

\begin{rem}\label{CI Section - Ideal Setting}
If $e=1$ in \Cref{CI Ring Setting}, then $E$ is isomorphic to a perfect $R$-ideal $I$ of grade two. Moreover, as $\varphi$ consists of linear entries, by the Hilbert-Burch theorem \cite[20.15]{Eisenbud} $I$ is minimally generated by $\mu(I)=d+1$ homogeneous forms of degree $d$.
\end{rem}

As before, we begin by considering the symmetric algebra of $E$. However, we first introduce notation amenable to the setting of \Cref{Duality Corollary}.

\begin{notat}\label{CI Ring Section - Notation}
With the assumptions of \Cref{CI Ring Setting}, write $T=k[y_1,\ldots,y_{d+e}]$ and let $B=S\otimes_k T \cong S[y_1,\ldots,y_{d+e}]$ with the bigrading given by $\bideg x_i=(1,0)$ and $\bideg y_i=(0,1)$. Writing $\overline{\,\cdot\,}$ to denote images modulo $(f,g)$, let $\psi$ be the $(d+e)\times d$ matrix with linear entries in $S$ such that $\overline{\psi} = \varphi$. Let $\L$ denote the $B$-ideal $\L = (\ell_1,\ldots,\ell_d,f,g)$ where $[\ell_1,\ldots,\ell_d]= [y_1\ldots y_{d+e}]\cdot \psi$.
\end{notat}

Clearly the matrix $\psi$ exists, and also it is unique as it has linear entries and $\deg f = \deg g =2$. Notice that $\overline{\L} = (\overline{\ell_1},\ldots,\overline{\ell_d})$ where $[\overline{\ell_1}\ldots\overline{\ell_d}] = [y_1\ldots y_{d+e}]\cdot \varphi$. Hence $\overline{\L}$ is the defining ideal of $\S(E)$, as a quotient of $R[y_1,\ldots,y_{d+e}]$.  By \Cref{S(E) Complete Intersection Criteria}, we have that $\overline{\ell_1},\ldots,\overline{\ell_d}$ is a regular sequence in this ring.

Since $R$ is not a polynomial ring, we cannot apply \Cref{Duality Theorem} with $R[y_1,\ldots,y_{d+e}]$ and $\overline{\L}$. However, note that
$$\S(E) = R[y_1,\ldots,y_{d+e}]/ \overline{\L} \cong B/ \L $$
and it follows that $\L$ is generated by a $B$-regular sequence. With this, $\S(E)$ is a complete intersection ring and can be realized as a quotient of a bigraded polynomial ring, hence we use this description of $\S(E)\cong B/\L$ with \Cref{Duality Theorem} and \Cref{Duality Corollary}.

\begin{rem}\label{CI Section - Defining Ideals}
As mentioned, we proceed with the isomorphism $\S(E)\cong B/ \L$, using $\L$ as the defining ideal of $\S(E)$, in this sense. Likewise, we may also update the ideals defining $\R(E)$. Differing from the convention in \Cref{Preliminaries}, write $\J$ to denote the kernel of the composition
$$B=S[y_1,\ldots,y_{d+e}] \longrightarrow R[y_1,\ldots,y_{d+e}] \longrightarrow \R(E)$$
where the first map quotients by $(f,g)$ and the second is the natural map. Writing $\A = \J/\L$, from \Cref{S(E) Complete Intersection Criteria} it follows that $\J=\L:\m^\infty$ and so $\A = H_\m^0(\S(E))$, where $\m = (x_1,\ldots,x_{d+2})$.
\end{rem}

The adjustment of the defining ideals $\L$, $\A$, and $\J$ is not completely novel and has also been used in \cite{Weaver23,Weaver24}. Recall that the generators of $\L = (\ell_1,\ldots,\ell_d,f,g)$ form a $B$-regular sequence. Moreover, $\ell_1,\ldots,\ell_d$ have bidegree $(1,1)$ and $f$ and $g$ have bidegree $(2,0)$. Hence one can see the similarity of $\S(E)$ in this setting to the symmetric algebra of \Cref{AALP Section}.

With the grading convention in \Cref{T-Grading}, we may produce $T$-resolutions of the graded components of $\S(E)\cong B/ \L$ from the bigraded Koszul complex of $\ell_1,\ldots,\ell_d,f,g$ following \Cref{Koszul}.

\begin{prop}\label{CI Section - Ranks and Complexes}
Adopt the assumptions of \Cref{CI Ring Setting} and write $\S(E)\cong B/ \L$ as above. With $\delta$ and $\tau$ as in \Cref{Duality Theorem}, we have $\delta =2$, $\tau = d$, and $\S(E)_0$ is free of rank one. Additionally $\rk \S(E)_1=2$ and $\rk \S(E)_2=1$.
\end{prop}

\begin{proof}
The first statement is clear. To verify the claim on the $T$-modules $\S(E)_1$ and $\S(E)_2$, we construct their graded $T$-resolutions following \Cref{Koszul} and apply \Cref{BE-AcyclicityCriteria} to compute their ranks. 

\begin{enumerate}[itemsep=0.5mm]
    \item[(a)] For the resolution of $\S(E)_1$, we have 
    
\begin{equation}\label{CIS1Resolution}
  \SelectTips{cm}{}
\xymatrix{0 \ar[r]&\overset{d}{\bigoplus} T(-1)^{\binom{d+1}{d+1}} \ar^-{\sigma}[r] & T^{\binom{d+2}{d+1}} \ar[r] & \S(E)_1 \ar[r]& 0}  
\end{equation}
and so $\rk \S(E)_1= (d+2)- d = 2$.

\item[(b)] For the resolution of $\S(E)_2$, we have 
\begin{equation}\label{CIS2Resolution}
\SelectTips{cm}{}
\xymatrix@C=18pt{0 \ar[r]&\overset{\binom{d}{2}}{\bigoplus} T(-2)^{\binom{d+1}{d+1}} \ar^-{\gamma}[r] & \overset{d}{\bigoplus} T(-1)^{\binom{d+2}{d+1}} \oplus \overset{2}{\bigoplus} T^{\binom{d+1}{d+1}} \ar^-{\rho}[r] & T^{\binom{d+3}{d+1}} \ar[r] & \S(E)_2 \ar[r]& 0}
\end{equation}
and so $\rk \S(E)_2 =\binom{d+3}{d+1}-d(d+2)-2+\binom{d}{2} =1$.\qedhere
\end{enumerate}
\end{proof}

Recall that (\ref{CIS1Resolution}) is the graded strand of the bigraded Koszul complex of $\ell_1,\ldots, \ell_d,f,g$ in degree 1. Choosing the standard monomial bases with the lexicographic order, $\sigma$ may be realized as the $(d+2)\times d$ matrix with linear entries in $T$ such that  $[\ell_1,\ldots,\ell_d] = [x_1\ldots x_{d+2}]\cdot \sigma$. Hence $\sigma$ may be identified with the Jacobian dual of $\psi$, $\sigma = \B(\psi)$, where $\psi$ is the matrix in \Cref{CI Ring Section - Notation}.

\begin{rem}\label{CI Section - I(sigma) Height}
With the matrix $\sigma$ as in (\ref{CIS1Resolution}), we have $\hgt I(\sigma) \geq 2$.
\end{rem}

\begin{proof}
As noted, $\sigma$ can be realized as the transition matrix with $[\ell_1\ldots \ell_d] = [x_1\ldots x_{d+2}]\cdot \sigma$. As $\sigma$ consists of entries in $T$, we see that $(\ell_1,\ldots,\ell_{d})$ is the defining ideal of the symmetric algebra of $\coker \sigma =\S(E)_1$. Thus by \cite[6.6, 6.8]{HSV83} it is enough to show that this ring is a domain. As $\psi$ consists of linear entries in $S$ and $\sigma$ consists of entries in $T$, there is an isomorphism of symmetric algebras $\S_S(M) \cong \S_T(\S(E)_1)$ where $M=\coker \psi$. Thus it suffices to show that $\S(M)$ is a domain, or rather that $M$ is an $S$-module of linear type.

Notice that $d= \rk \varphi \leq \rk \psi \leq d$, and so $\rk \psi =d$. Thus the complex of $S$-modules
$$0 \longrightarrow S^d \overset{\psi}{\longrightarrow} S^{d+e} \longrightarrow M \longrightarrow 0$$ is exact by \Cref{BE-AcyclicityCriteria}, and so $M$ is an $S$-module of projective dimension one. Notice that $M$ satisfies $G_d$ since $E$ satisfies this condition. Indeed, $\hgt {\rm Fitt}_i(M) \geq \hgt {\rm Fitt}_i(E) \geq i-e+2$ for all $e\leq i\leq d+e-2$, as $\overline{\psi} = \varphi$ and height can only decrease modulo $(f,g)$. Additionally, notice that $\fitt_{d+e-1}(M) = I_1(\psi) = (x_1,\ldots,x_{d+2})$, which has height $d+2\geq d+1$. Since $\mu(M) = d+e$, for any $i\geq d+e$ we see that $\fitt_i(M)$ is the unit ideal. Thus $M$ satisfies $G_\infty$, and is hence of linear type by \cite[Prop. 3 and 4]{Avramov81}.
\end{proof}

Before we state the main result of this section, we recall the notion of a \textit{modified} Jacobian dual matrix introduced in \cite{Weaver23}. This matrix is an extension of the usual Jacobian dual and will inevitably be the matrix in \Cref{A-delta generator}, in this setting.

\begin{defn}\label{MJD defn}
With the assumptions of \Cref{CI Ring Setting}, let $\psi$ be as in \Cref{CI Ring Section - Notation}. We define a \textit{modified Jacobian dual} of $\psi$ to be a $(d+2)\times(d+2)$ matrix $\B = [\B(\psi)\,|\,\partial f\, |\,\partial g]$ where $\B(\psi)$ is the Jacobian dual of $\psi$, with respect to $x_1,\ldots,x_{d+2}$, and $\partial f$ and $\partial g$ are columns with linear entries in $S$ such that $f= [x_1\ldots x_{d+2}]\cdot \partial f$ and $g= [x_1\ldots x_{d+2}]\cdot \partial g$.
\end{defn}

Note that whereas the Jacobian dual $\B(\psi)$ is unique, a modified Jacobian dual $\B$ is not, as there are many choices for the columns $\partial f$ and $\partial g$. 
However, there is a natural candidate for this matrix.

\begin{rem}\label{MJD natural candidate}
As the notation suggests, there is a natural choice for the columns $\partial f$ and $\partial g$ using differentials, if $k$ is a field with $\chr k \neq 2$. Indeed, these columns may be chosen to consist of (unit multiples of) the partial derivatives of $f$ and $g$, using the well-known Euler formula
\begin{equation}\label{Euler Formula}
 2f = \sum_{i=1}^{d+2} \frac{\partial f}{\partial x_i} x_i, \quad \quad 2g = \sum_{i=1}^{d+2} \frac{\partial g}{\partial x_i} x_i,   
\end{equation}
as $f$ and $g$ are homogeneous of degree 2. 
\end{rem}

\begin{thm}\label{CI Section - Main Result}
With $E$ as in \Cref{CI Ring Setting} and $\A$ the defining ideal of $\R(E)$ in the setting of \Cref{CI Section - Defining Ideals}, we have the following.

\begin{enumerate}
    \item[(a)] $\A_2$ is generated as $\A_2 = \langle\det \B\rangle$ where $\B$ is a modified Jacobian dual of $\psi$. Moreover, this is an equation of bidegree $(2,d)$.

    \item[(b)] If $\hgt I(\sigma) = 3$, then $\A_1$ is minimally generated by $d$ equations of bidegree $(1,2d)$.

    \item[(c)] $\A_0$ is minimally generated by one equation. If $\hgt I(\rho) \geq 2$, then this generator has bidegree $(0,4d)$.
\end{enumerate}
\end{thm}

\begin{proof}
Part (a) follows from \Cref{Duality Corollary} and \Cref{A-delta generator}, noting that $[\ell_1 \ldots \ell_d \, f\,  g] = [x_1 \ldots x_{d+2}] \cdot \B$. For the first assertion of (c), the same argument as in the proof of \Cref{AALP Section Main Result} shows that $E$ has analytic spread $\ell(E) = d+e-1$. As this is the dimension of the special fiber ring $\F(E)$, it follows that $\A_0$ is a cyclic $T$-module.

Notice that $\sigma$ is a $(d+2)\times d$ matrix, hence $\hgt I(\sigma)\leq 3$ by \cite[Thm. 1]{EN62}. Thus (b) and the second assertion of (c) follow from \Cref{KMcomplex} along with \Cref{CI Section - Ranks and Complexes}, once the degree shift of the map $a_1$ has been computed for resolutions (\ref{CIS1Resolution}) and (\ref{CIS2Resolution}).

\begin{enumerate}[itemsep=0.5mm]
    \item[(a)] For the $T$-resolution of $S(E)_1$ in (\ref{CIS1Resolution}), notice that $a_1 = \bigwedge^{d} \sigma$, following  \Cref{BEmultipliers}. Hence
    $$a_1\,:\,\bigwedge^{d} \big(T(-1)^{d}\big ) \cong T(-d) \longrightarrow  \bigwedge^{d} (T^{d+2})$$
    and so $s = d$.

    \item[(b)] For the $T$-resolution of $S(E)_2$ in (\ref{CIS2Resolution}), we begin with $a_2$ as this resolution has length two and the construction of the maps in \Cref{BEmultipliers} is recursive. Note that 
        $$a_2 = \bigwedge^{\binom{d}{2}} \gamma\,:\, \bigwedge^{\binom{d}{2}} T(-2)^{\binom{d}{2}} \longrightarrow \bigwedge^{\binom{d}{2}}\big(T(-1)^{d(d+2)}\oplus T^2\big)$$
       and it can be seen that the entries in a matrix representation of $a_2$ have degree $\binom{d}{2}$. By \Cref{BEmultipliers} we have the commutative diagram
        \[
    \SelectTips{cm}{}
    \xymatrix{\bigwedge^{r} \big(T(-1)^{d(d+2)}\oplus T^2\big) \ar_-{a_{2}^*}[drr] \ar^-{\bigwedge^{r}\rho}[rrrr] & & & & \bigwedge^{r} T^{\binom{d+3}{d+1}}\\
    & & T(-s) \ar_-{a_1}[urr]& & }
    \]
    where $r = \rk \rho = d(d+2)+2 - \binom{d}{2}$ by \Cref{BE-AcyclicityCriteria}. Notice that the entries in a matrix representation of $\bigwedge^{r}\rho$ have degree $r-2 = d(d+2)- \binom{d}{2}$ as $\rho$ consists of linear entries and two columns with units. Hence $s= (r-2) - \binom{d}{2} = d(d+2) - d(d-1) = 3d$.
\end{enumerate}
The claim now follows from \Cref{KMcomplex}, noting that $\tau=d$ by \Cref{CI Section - Ranks and Complexes}.
\end{proof}

\begin{rem}\label{CI Section - Ideal Remark}
Similar to the observation made in \Cref{AALP Section}, resolutions (\ref{CIS1Resolution}) and (\ref{CIS2Resolution}) do not depend on the rank of $E$. Thus if $e=1$, one has the same result for the Rees algebra of an ideal $I$ as in \Cref{CI Section - Ideal Setting}.
\end{rem}




\begin{rem}\label{CI section expected A}
If the conditions of \Cref{CI Section - Main Result} are met, then it is understood how the defining ideals $\A$ and $\J$ of \Cref{CI Section - Defining Ideals} are generated. Indeed, if both $\hgt I(\sigma) = 3$ and $\hgt I(\rho) \geq 2$ in \Cref{CI Section - Main Result}, then $\J$ is generated as 

\begin{center}
 \begin{tabular}{|c |c |c |} 
 \hline
 & Bidegree & Number of generators \\ [0.5ex] 
 \hline
 $\ell_1,\ldots,\ell_{d}$ & $(1,1)$ & $d$ \\ 
 \hline
  $f,g$ & $(2,0)$ & $2$ \\
 \hline
 $\A_0$ & $(0,4d)$ &  $1$\\
 \hline
 $\A_1$ & $(1,2d)$& $d+2$ \\
 \hline
 $\A_2=\langle\det \B\rangle$ & $(2,d)$ & $1$ \\
 \hline
\end{tabular}
\end{center}
where $\B$ is a modified Jacobian dual of $\psi$. In particular, $\J$ is minimally generated by $2d+6$ elements and $\A$ is minimally generated by $d+4$ elements.
\end{rem}

As before, we refer to the behavior in \Cref{CI section expected A} as the predicted form of $\A$ and $\J$ in this setting. However, we present an example to illustrate that this behavior is not always observed, even in the case of ideals. We note that this phenomenon is frequently observed, and it is suspected that the behavior in \Cref{CI section expected A} always occurs if the entries of $\varphi$, and also $f$ and $g$, are sufficiently general.

\begin{ex}\label{CI Section - Example I(rho) Height one}
Let $S={\mathbb Q}[x_1,x_2,x_3,x_4]$ and let $R=S/(x_1^2,x_2^2)$. Consider the matrix $\varphi$ with linear entries in $R$ and its corresponding matrix $\psi$ with linear entries in $S$,
\[
\varphi = \begin{bmatrix}
    \overline{x_1} &\overline{x_3}\\
    \overline{x_3}&\overline{x_4}\\
    \overline{x_4}&\overline{x_2}
\end{bmatrix} \quad\quad \text{and}\quad\quad 
\psi =\begin{bmatrix}
    x_1 &x_3\\
    x_3&x_4\\
    x_4&x_2
\end{bmatrix}
\]
and let $I =I_2(\varphi)$. Computations through \textit{Macaulay2} \cite{Macaulay2} show that $\hgt I =2$, hence $I$ is perfect of grade 2 \cite[20.15]{Eisenbud} and satisfies $G_2$. Thus the assumptions of \Cref{CI Section - Ideal Setting} and \Cref{CI Ring Setting} are met.

Choosing the standard monomial bases for the free modules involved, a matrix representative of the map $\rho$ in the resolution (\ref{CIS2Resolution}) of $\S(E)_2$ is
\[
\rho =\begin{bmatrix}
    
      1&0&{y}_{1}&0&0&0&0&0&0&0\\
      0&0&0&{y}_{1}&0&0&{y}_{3}&0&0&0\\
      0&0&{y}_{2}&0&{y}_{1}&0&{y}_{1}&0&0&0\\
      0&0&{y}_{3}&0&0&{y}_{1}&{y}_{2}&0&0&0\\
      0&1&0&0&0&0&0&{y}_{3}&0&0\\
      0&0&0&{y}_{2}&0&0&0&{y}_{1}&{y}_{3}&0\\
      0&0&0&{y}_{3}&0&0&0&{y}_{2}&0&{y}_{3}\\
      0&0&0&0&{y}_{2}&0&0&0&{y}_{1}&0\\
      0&0&0&0&{y}_{3}&{y}_{2}&0&0&{y}_{2}&{y}_{1}\\
      0&0&0&0&0&{y}_{3}&0&0&0&{y}_{2}
      \end{bmatrix}
\]
and \Cref{BE-AcyclicityCriteria} and (\ref{CIS2Resolution}) show that $\rk \rho =9$. However, $\hgt I(\rho) =1$ and so the conditions of \Cref{CI Section - Main Result} are not satisfied. Further computations show that $\A_0$ is generated by an equation of bidegree $(0,6)$, differing from the predicted behavior in \Cref{CI section expected A}. However, with $\sigma$ as in (\ref{CIS1Resolution}) we have $\hgt I(\sigma)=3$ in this example, and so $\A_1$ is generated by four equations of bidegree $(1,4)$ by \Cref{CI Section - Ideal Remark} and \Cref{CI Section - Main Result}. 
\end{ex}

One possibility for why \Cref{CI Section - Example I(rho) Height one} failed to produce the expected behavior in \Cref{CI section expected A} is that the ring $R$ is not reduced. It is curious if one obtains the predicted behavior if $R$ is reduced or a domain. We investigate this possibility in the proceeding section, under the assumption that $R$ is a \textit{normal} complete intersection domain.

\begin{quest}
Noting that $\A_2$ is generated by the determinant of a modified Jacobian dual $\B$, one might ask if the remaining equations of $\A$ can be described by similar constructions. This was shown to be the case for Rees rings of perfect ideals of grade two and modules of projective dimension one over hypersurface rings in \cite{Weaver23}. This was also shown to be the case in \cite{Weaver24} for Rees algebras of perfect Gorenstein ideals of grade three in hypersurface rings. 
\end{quest}


\section{Applications to Tangent algebras}\label{Tangent Algebras Section}

In this final section, we discuss applications of the results from \Cref{CI Ring Section} to the Rees ring of the module of K\"ahler differentials $\Omega_{R/k}$ of a complete intersection ring $R$ defined by quadrics. Whereas this situation possesses great significance, certain aspects of our setting are limited due to the relationship between $\Omega_{R/k}$ and the ring $R$. However, other aspects appear more natural as many conditions on $\Omega_{R/k}$ are implied by assumptions on $R$.

We refer the reader to \cite{Kunz} for the necessary background material on the module of differentials. We require few technical aspects of the subject for our treatment, hence much of the preliminary material is omitted. Our setting throughout is as follows.

\begin{set}\label{ModDiffSetting}
    Let $S=k[x_1,x_2,x_3,x_4]$ for $k$ a field of characteristic zero and let $R= S/(f,g)$ where $f,g$ is a $S$-regular sequence of homogeneous polynomials with $\deg f= \deg g=2$. Assume that $R$ is a normal ring and let $\Omega_{R/k}$ denote the module of K\"ahler differentials of $R$ over $k$. 
\end{set}

We note that many of the results presented here hold if $k$ is a perfect field with $\chr k \neq 2$. For simplicity however, we consider the situation above.

\begin{rem}\label{Normal Domain}
Notice that, since $R$ is a normal complete intersection ring, the ideal $(f,g)$ is prime by Hartshorne's connectedness lemma \cite{Hartshorne62}. Hence $R$ is a domain (see also \cite[6.15 -- 6.17]{VasconcelosBook2}). 
\end{rem}

With the properties of the ring $R$ in \Cref{ModDiffSetting}, we note that $\Omega_{R/k}$ satisfies the assumptions of \Cref{CI Ring Section}. Indeed, many of the required conditions are consequences of the well-known Jacobian criterion \cite[7.2]{Kunz}.

\begin{prop}\label{Omega Properties}
With the conditions of \Cref{ModDiffSetting}, $\Omega_{R/k}$ satisfies the assumptions of \Cref{CI Ring Setting}.
\end{prop}

\begin{proof}
Since $R$ is a complete intersection ring, it is well known that $\Omega_{R/k}$ is an $R$-module of projective dimension at most one \cite{Ferrand67} and hence exactly one as $R$ is not regular \cite[7.2]{Kunz}. Thus there is a short exact sequence
\begin{equation}\label{PresentationOfOmega}
0 \longrightarrow R^2 \overset{\overline{\Theta}}{\longrightarrow} R^4 \longrightarrow \Omega_{R/k} \longrightarrow 0
\end{equation}
where $\Theta$ is the transpose of the Jacobian matrix of $f$ and $g$ and $\overline{\,\cdot\,}$ denotes images modulo $(f,g)$ \cite[4.19]{Kunz}. Since $R$ is normal, it satisfies Serre's condition $R_1$ and is hence regular in codimension 1. In particular, it follows that $\hgt \fitt_2(\Omega_{R/k}) \geq 2$, or equivalently  that $\Omega_{R/k}$ is locally free in codimension 1 \cite[7.2]{Kunz}, or also equivalently that $\Omega_{R/k}$ satisfies $G_2$. With this, it also follows that $\Omega_{R/k}$ is torsion-free.

For the rank and number of generators of $\Omega_{R/k}$, we refer to 
\cite[ch. 7]{Kunz}. With the conditions of \Cref{ModDiffSetting}, the rank of $\Omega_{R/k}$ agrees with the dimension of $R$, $\rk \Omega_{R/k} = \dim R =2$. Moreover, the minimal number of generators of $\Omega_{R/k}$ agrees with the embedding dimension of $R$, $\mu(\Omega_{R/k}) = \edim R =4$. With this, we see that $\mu(\Omega_{R/k}) = \dim R + \rk \Omega_{R/k}$.

It remains to be shown that $I_1(\overline{\Theta}) = \overline{(x_1,\ldots,x_4)}$, the homogeneous maximal ideal of $R$. Notice that $\overline{\Theta}$ consists of linear entries in $R$ since $f$ and $g$ are quadrics. Recall that $\hgt \fitt_2(\Omega_{R/k}) \geq 2$ and note that this ideal is $I_2(\overline{\Theta})$, hence the containment $I_2(\overline{\Theta})\subseteq I_1(\overline{\Theta})$ shows that $\hgt I_1(\overline{\Theta})=2$. Thus the claim will follow once it has been shown that $I_1(\overline{\Theta})$ is a prime ideal. However, this follows as $I_1(\Theta)$ is prime, as it is generated by linear forms, and since $f,g \in I_1(\Theta)$, which can be seen from the Euler formula (\ref{Euler Formula}). 
\end{proof}

The ring $R$ is assumed to be two-dimensional in \Cref{ModDiffSetting} in order to satisfy the generation condition in \Cref{CI Ring Setting}. Indeed, the proof of \Cref{Omega Properties} shows that we require $\edim(R) = 2\dim R$. However, if $R$ is a quotient by a regular sequence of length two, this can only happen if $\dim R=2$.


As before, we intend to study the defining ideal $\A$ of the Rees ring $\R(\Omega_{R/k})$ using the symmetric algebra $\S(\Omega_{R/k})$. As in \Cref{CI Ring Section}, we introduce the notation necessary for its treatment. We adopt conventions similar to those in \Cref{CI Ring Section - Notation} and \Cref{CI Section - Defining Ideals}, and restate them for clarity.

\begin{notat}\label{SymmetricAlgebraOfOmega}
With the assumptions of \Cref{ModDiffSetting}, write $T=k[y_1,\ldots,y_{4}]$ and let $B=S\otimes_k T \cong S[y_1,\ldots,y_{4}]$. 
As in the proof of \Cref{Omega Properties}, let $\Theta$ denote the Jacobian matrix of $f$ and $g$ and let $\L$ denote the $B$-ideal $\L = (\ell_1,\ell_2,f,g)$, where $[\ell_1 \, \ell_2]= [y_1\ldots y_{4}]\cdot \Theta$.
\end{notat}

As before, this allows us to update the notion of the defining ideals of $\S(\Omega_{R/k})$ and $\R(\Omega_{R/k})$. Following \Cref{CI Section - Defining Ideals}, we have $\S(\Omega_{R/k}) \cong B/ \L$ and we take $\J$ to be the kernel of the composition
$$B \longrightarrow R[y_1,\ldots,y_{4}] \longrightarrow \R( \Omega_{R/k})$$
so that $\R(\Omega_{R/k}) \cong B/\J$. Following \Cref{S(E) Complete Intersection Criteria} and \Cref{CI Section - Defining Ideals}, we have $\J= \L:\m^\infty$ and also $\A = \J/\L \cong H_\m^0(\S(\Omega_{R/k}))$, where $\m = (x_1,x_2,x_3,x_4)$.

\begin{rem}\label{modDiff complexes}
Recall that the generators of $\L$ form a $B$-regular sequence, and so $\S(\Omega_{R/k})\cong B/\L$ is a complete intersection ring. Moreover, by \Cref{CI Section - Ranks and Complexes} we have $\delta=\tau=2$ and we need only analyze the $T$-resolutions of $\S(\Omega_{R/k})_1$ and $\S(\Omega_{R/k})_2$. From (\ref{CIS1Resolution}) and (\ref{CIS2Resolution}) we have the following.


\begin{enumerate}[itemsep=0.5mm]
    \item[(a)] The $T$-resolution of $\S(\Omega_{R/k})_1$ is
\begin{equation}\label{S(Omega)Resolution1}
  \SelectTips{cm}{}
\xymatrix{0 \ar[r]&T(-1)^2 \ar^-{\sigma}[r] & T^4 \ar[r] & \S(\Omega_{R/k})_1 \ar[r]& 0.}  
\end{equation}

    \item[(b)] The $T$-resolution of $\S(\Omega_{R/k})_2$ is \begin{equation}\label{S(Omega)Resolution2}
 \SelectTips{cm}{}
\xymatrix@C=18pt{0 \ar[r]& T(-2) \ar^-{\gamma}[r] &  T(-1)^8 \oplus  T^2 \ar^-{\rho}[r] & T^{10} \ar[r] & \S(\Omega_{R/k})_2 \ar[r]& 0.}   
\end{equation}
\end{enumerate}
\end{rem}

As noted prior to \Cref{CI Section - I(sigma) Height}, $\sigma$ may be identified with the Jacobian dual of $\Theta$, $\sigma = \B(\Theta)$, with respect to $x_1,\ldots,x_4$. As the terminology suggests, a Jacobian dual may be seen as a Jacobian matrix. Moreover, a particularly interesting phenomenon occurs for the Jacobian dual of a Jacobian matrix of quadrics. Before we reveal this aspect, we introduce two maps allowing one to toggle between objects over the rings $S$ and $T$.

\begin{defn}
With $S=k[x_1,x_2,x_3,x_4]$ and $T=k[y_1,y_2,y_3,y_4]$ as before, let $\Phi$ denote the isomorphism $\Phi: T\overset{\sim}{\rightarrow} S$ given by $\Phi(y_i) =x_i$. Moreover, let $\Psi$ denote its inverse $\Psi: S\overset{\sim}{\rightarrow} T$ given by $\Psi(x_i)=y_i$.
\end{defn}

In particular, one may apply the isomorphism $\Phi$ in order to treat the maps in (\ref{S(Omega)Resolution1}) and (\ref{S(Omega)Resolution2}) as maps between $S$-modules. One may then apply its inverse $\Psi$ to transfer any discovered information back. We provide a short lemma applying this to the map $\sigma$ and offer a proof by simple calculus.

\begin{lemma}\label{Jacobians}
With $\sigma$ the map in (\ref{S(Omega)Resolution1}), we may take $\sigma = \B(\Theta)$ as a matrix representation. Additionally, 
we have $\B(\Theta) = \Psi(\Theta)$ and $\Theta = \Phi(\B(\Theta))$.
\end{lemma}

\begin{proof}
The first assertion has already been verified, hence we proceed with the second. Recall that, as $\Theta$ consists of linear entries in $S$, $\B(\Theta)$ is the \textit{unique} matrix with linear entries in $T$ such that 
\begin{equation}\label{JDThetaEquation}[y_1\ldots y_4]\cdot \Theta 
= [x_1\ldots x_4]\cdot \B(\Theta).
\end{equation}
Hence we need only show that 
\begin{equation}\label{nu(Theta)equation}
   [y_1\ldots y_4]\cdot \Theta 
= [x_1\ldots x_4]\cdot \Psi(\Theta) 
\end{equation}
as well. 

Since $\Theta$ is the Jacobian matrix of $f$ and $g$ with respect to $x_1,\ldots,x_4$, it follows that $\Psi(\Theta)$ is the Jacobian matrix of $\Psi(f)$ and $\Psi(g)$, with respect to $y_1,\ldots,y_4$. Thus it suffices to show that 
\begin{equation}\label{Derivative Sum}
    \sum_{t=1}^4 \frac{\partial h}{\partial x_t} y_t = \sum_{t=1}^4 \frac{\partial (\Psi(h))}{\partial y_t} x_t
\end{equation}
for any homogeneous quadric $h\in S$ to conclude the equality in (\ref{nu(Theta)equation}). However, by linearity of differentiation, it suffices to check that (\ref{Derivative Sum}) holds for any monomial $h=x_ix_j$, which is easily verified. Thus $\Psi(\Theta)=\B(\Theta)$ and applying the inverse $\Phi$ shows that $\Theta = \Phi(\B(\Theta))$ as well.
\end{proof}

We remark that \Cref{Jacobians} requires only that the entries of $\Theta$ are linear, hence it may easily be extended.  
In particular, the Jacobian dual of a Jacobian matrix of quadrics is essentially the same Jacobian matrix.

\begin{ex}\label{Jacobians Example}
Let $S={\mathbb Q}[x_1,x_2,x_3,x_4]$ and let $f= x_1^2+x_2x_3 +x_4^2$ and $g=x_2^2+x_1x_4+x_3^2$. The Jacobian matrix $\Theta$ of $f$ and $g$ and its Jacobian dual $\B(\Theta)$ are
\[
\Theta = \begin{bmatrix}
    2x_1& x_4\\
    x_3&2x_2\\
    x_2&2x_3\\
    2x_4&x_1
\end{bmatrix}\quad \quad \text{and} \quad \quad
\B(\Theta) = \begin{bmatrix}
    2y_1& y_4\\
    y_3&2y_2\\
    y_2&2y_3\\
    2y_4&y_1
\end{bmatrix}
\]
which can be seen from the equation $[y_1\ldots y_4]\cdot \Theta = [x_1\ldots x_4]\cdot \B(\Theta)$. Hence $\B(\Theta)$ is also the Jacobian matrix of $\Psi(f)= y_1^2+y_2y_3 +y_4^2$ and $\Psi(g)=y_2^2+y_1y_4+y_3^2$, by \Cref{Jacobians}.
\end{ex}

As a consequence of \Cref{Jacobians}, properties of $\sigma$ in (\ref{S(Omega)Resolution1}) can be deduced from properties of $\Theta$, heights of ideals of minors in particular. With this, we present the main result of this section.

\begin{thm}\label{ModDiff - main result}
With the assumptions of \Cref{ModDiffSetting} and $\A$ as before, we have the following.
    \begin{enumerate}
    \item[(a)] $\A_2$ is generated as $\A_2=\langle\det \B\rangle$ where $\B = [\B(\Theta)\,|\,\Theta]$. Moreover, this equation has bidegree $(2,2)$.

    \item[(b)] $\A_1$ is minimally generated by four equations of bidegree $(1,4)$. 

    \item[(c)] $\A_0$ is minimally generated by one equation. If $\hgt I(\rho) \geq 2$, then this generator has bidegree $(0,8)$.
\end{enumerate}
In particular, $\A$ is generated by six elements and the ideal $\J$ is generated by ten elements.
\end{thm}

\begin{proof}
By \Cref{Omega Properties}, $\Omega_{R/k}$ satisfies the assumptions of \Cref{CI Section - Main Result}, hence assertion (c) follows immediately. By \Cref{CI Section - Main Result} we have that $\A_2$ is generated by the determinant of a modified Jacobian dual of $\Theta$. Thus part (a) follows from \Cref{A-delta generator} and \Cref{MJD natural candidate}, noting that $2$ is a unit.

For part (b), by \Cref{CI Section - Main Result} it suffices to show that $\hgt I_2(\sigma) =3$. 
Recall that $\Phi$ and $\Psi$ are isomorphisms, hence by \Cref{Jacobians} we have $\hgt I_2(\sigma) = \hgt I_2(\B(\Theta)) = \hgt I_2(\Theta) \leq 3$, where the inequality follows from \cite[Thm. 1]{EN62}. However, by \cite[2.4]{Yoshino85} we have $\hgt I_2(\sigma)=\hgt I_2(\Theta) \geq 3$, as $R = S/(f,g)$ is a complete intersection ring satisfying Serre's condition $R_1$. 
\end{proof}

As noted, $\A_2$ is generated by the determinant of the modified Jacobian dual $\B = [\B(\Theta)\,|\,\Theta]$. It is particularly interesting to note that this matrix is a \textit{total} Jacobian of the generators of $\L=(\ell_1,\ell_2,f,g)$, with respect to $x_1,\ldots,x_4$. Moreover, recall from \Cref{Jacobians} that $\B(\Theta) = \Psi(\Theta)$, hence there is a symmetry within this matrix. For instance, as a continuation of \Cref{Jacobians Example}, we have the following.

\begin{ex}
With $S$ and $f,g$ as in \Cref{Jacobians Example}, let $R=S/(f,g)$ and $\Omega_{R/k}$ its module of differentials. With $\A$ as above, $\A_2$ is generated by the determinant of 
\[
[\B(\Theta)\,|\,\Theta]=  [\Psi(\Theta)\,|\,\Theta]= \begin{bmatrix}
    2y_1& y_4&2x_1& x_4\\
    y_3&2y_2&x_3&2x_2 \\
    y_2&2y_3&x_2&2x_3\\
     2y_4&y_1&2x_4&x_1 
\end{bmatrix}
\]
following \Cref{ModDiff - main result} and \Cref{Jacobians}.
\end{ex}

Although the number of generators of $\A$ is known in \Cref{ModDiff - main result}, we lack certainty on the bidegree of the generator of $\A_0$. However, we provide the following conjecture.

\begin{conj}\label{Height Conjecture}
With the assumptions of \Cref{ModDiff - main result} and $\rho$ the map in (\ref{S(Omega)Resolution2}), we claim that $\hgt I(\rho)\geq 2$. Then by \Cref{ModDiff - main result},  $\A_0$ is generated by an equation of bidegree $(0,8)$.
\end{conj}

Similar to \Cref{Jacobians} and the proof of \Cref{ModDiff - main result}, one may apply the map $\Phi$ to the map $\rho$ in (\ref{S(Omega)Resolution2}) and investigate the codimension of $I(\Phi(\rho))$ instead. As in \Cref{Jacobians}, the entries of $\Phi(\rho)$ can be described by partial derivatives of $f$ and $g$. Computations through \textit{Macaulay2} \cite{Macaulay2} suggest that this ideal of minors of $\Phi(\rho)$ has height exactly 2, and also provides a guess as to some of the minimal primes of this ideal. Regardless, we save further investigation and optimistic confirmation of \Cref{Height Conjecture} for a later article.


\subsection{Higher dimension}

As previously noted, we considered a two-dimensional ring in \Cref{ModDiffSetting} so that $\edim R=2\dim R$, and so $\Omega_{R/k}$ satisfies the generation condition in \Cref{CI Section - Main Result}. It is curious as to what can be said if one maintains this condition, but allows for higher dimension $d$.

\begin{set}\label{General ModDiff Setting}
Let $S=k[x_1,\ldots,x_{2d}]$ for $d\geq 2$ and $k$ a field of characteristic zero. Let $f_1,\ldots,f_d$ be a homogeneous regular sequence with each $\deg f_i =2$ and let $R=S/(f_1,\ldots,f_d)$. Let $\Omega_{R/k}$ denote the module of differentials of $R$ over $k$, and assume that $\Omega_{R/k}$ satisfies $G_d$.
\end{set}

Repeating the argument in \Cref{Omega Properties} and \Cref{Normal Domain}, notice that $R$ is a normal domain here as $\Omega_{R/k}$ satisfies $G_2$. The full condition $G_d$ must be assumed here as is not implied from normality alone, unless $d=2$, as in \Cref{ModDiffSetting}. However, this condition will be met if $R$ satisfies Serre's condition $R_{d-1}$, i.e. if $R$ is regular in codimension at most $d-1$ or, in other words, if $R$ is an \textit{isolated singularity}.


Similar arguments to the proof of \Cref{Omega Properties} show that $\Omega_{R/k}$ satisfies the assumptions of \Cref{S(E) Complete Intersection Criteria}, hence $\S(\Omega_{R/k})$ is a complete intersection ring. Moreover, we may write $\S(\Omega_{R/k}) \cong B/\L$ where $B=S[y_1,\ldots,y_{2d}]$ and $\L= (\ell_1,\ldots,\ell_d,f_1,\ldots,f_d)$ where $[\ell_1\ldots \ell_d] = [y_1\ldots y_{2d}] \cdot \Theta$. Here $\Theta$ is the transposed Jacobian matrix of $f_1,\ldots,f_d$, as before. One may apply \Cref{Duality Theorem} and \Cref{Duality Corollary} since $\S(\Omega_{R/k}) \cong B/\L$ is a complete intersection ring, however one cannot use \Cref{KMcomplex} to describe $\A_i$ for all $i$, as the ranks of the $T$-modules $\S(\Omega_{R/k})_i$ become too large. 


\begin{obs}\label{S(Omega)i Ranks in general}
With the assumptions of \Cref{General ModDiff Setting}, one finds that $\delta = \tau =d$. Moreover, The $T$-resolutions of $\S(\Omega_{R/k})_i$ can be computed using \Cref{Koszul}. From these resolutions and \Cref{BE-AcyclicityCriteria}, one observes that the rank of each $T$-module is $\rk \, \S(\Omega_{R/k})_i = \binom{\,d\,}{i}$ for $0\leq i\leq \delta =d$.
\end{obs}

Certainly this was the case in \Cref{modDiff complexes} when $d=2$. Moreover, this confirms that \Cref{KMcomplex} cannot be applied to study every component $\A_i$ when $d\geq 3$. Regardless, since both $\S(\Omega_{R/k})_0$ and $\S(\Omega_{R/k})_d$ have rank 1, perhaps one can determine the corresponding generators of $\A_d$ and $\A_0$. As before, one has a natural candidate for the generator of $\A_d$.

\begin{rem}
With the assumptions of \Cref{General ModDiff Setting}, let $\A$ be the ideal defining $\R(\Omega_{R/k})$ as a quotient of $\S(\Omega_{R/k})$. The $T$-module $\A_d$ is generated as $\A_d=\langle\det \B\rangle$ where $\B = [\B(\Theta)\,|\,\Theta]$, for $\B(\Theta)$ the Jacobian dual of $\Theta$ with respect to $x_1,\ldots,x_{2d}$. Moreover, this is an equation of bidegree $(d,d)$.
\end{rem}

\begin{proof}
    This follows similar to the proof of part (a) of \Cref{ModDiff - main result}.
\end{proof}

As noted prior to \Cref{Jacobians Example}, the statement of \Cref{Jacobians} may be extended to this setting, and there is a similar relationship between $\B(\Theta)$ and $\Theta$.

As we also have that $\rk \S(\Omega_{R/k})_d =1$, one may apply \Cref{KMcomplex} to the graded $T$-resolution of this module in order to study $\A_0$. Similar to the proof of \Cref{CI Section - Main Result}, $\A_0$ is a cyclic $T$-module and \Cref{KMcomplex} can be used to determine the bidegree of its generator, so long as the ideal of minors of the syzygy matrix of $\S(\Omega_{R/k})_d$ has height at least two. This is the topic of \Cref{Height Conjecture}, hence it may not be straightforward, however the author intends to investigate this in a subsequent article.


\section*{Acknowledgements}

The author would like to thank Youngsu Kim and Vivek Mukundan for their helpful comments and suggestions on an earlier draft of this paper. The use of \textit{Macaulay2} \cite{Macaulay2} was helpful in the preparation of this article, providing numerous examples of the results introduced here.



\begin{thebibliography}{99}


\bibitem{Avramov81} L.~Avramov, \textit{Complete intersections and symmetric algebras}, J. Algebra \textbf{73} (1981), 248--263. 
 
\bibitem{BM16} J.~A.~Boswell and V.~Mukundan, \textit{Rees algebras of almost linearly presented ideals}, J. Algebra \textbf{460} (2016), 102--127.




  


\bibitem{BH93} W. Bruns and J. Herzog, \textit{Cohen-Macaulay rings}, Cambridge Studies in Advanced Mathematics \textbf{39}, Cambridge University Press, Cambridge, 1993. 




\bibitem{BE73} D.A. Buchsbaum and D. Eisenbud, \textit{What makes a complex exact?}, J. Algebra \textbf{25} (1973), 259--268.

\bibitem{BE74} D.A. Buchsbaum and D. Eisenbud, \textit{Some structure theorems for finite free resolutions}, Advances in Math. \textbf{12} (1974), 84--139.


 





\bibitem{Costantini21} A.~Costantini, \textit{Cohen-Macaulay fiber cones and defining ideal of Rees algebras of modules}, Women in Commutative Algebra -- Proceedings of the 2019 WICA Workshop, Association for Women in Mathematics Series, vol. 29, Springer (2022). 



\bibitem{CPW23} A.~Costantini, E.~F.~Price, and M.~Weaver, \textit{On Rees algebras of linearly presented ideals and modules}, \texttt{arxiv:2308.16010}. (To appear in Collect. Math.)

\bibitem{CD20} A. Costantini and T. Dang, \textit{On the Cohen-Macaulay property of the Rees algebra of the modules of differentials}, Proc. Amer. Math. Soc. \textbf{150} (2022), 941--950 


\bibitem{CHW08} D.~Cox, J.~W. Hoffman, and H.~Wang, \textit{Syzygies and the {R}ees algebra}, J. Pure Appl. Algebra \textbf{212} (2008), 1787--1796.

\bibitem{EN62} J.~A.~Eagon and D.~G.~Northcott, \textit{Ideals defined by matrices and a certain complex associated with them}, Proc.~ Roy.~Soc. Ser. A \textbf{269} (1962), 188--204.

\bibitem{Eisenbud} D.~Eisenbud, \textit{Commutative algebra: with a view toward algebraic geometry}, Graduate Texts in Mathematics \textbf{150}, Springer-Verlag, New York, 1995. 


\bibitem{EHU03} D. Eisenbud, C. Huneke and B. Ulrich, \textit{What is the Rees algebra of a module?}, Proc. Amer. Math. Soc. \textbf{131} (2003), 701--708.


\bibitem{Ferrand67} D. Ferrand, \textit{Suite r\'{e}guli\`{e}re et intersection compl\`{e}te}, C. R. Acad. Sci. Paris S\'{e}r. A-B \textbf{264} (1967), 247--248.

\bibitem{Macaulay2} D. R. Grayson and M. E. Stillman, Macaulay2, a software system for research in algebraic geometry. Available at http://www.math.uiuc.edu/Macaulay2/

\bibitem{Harris} J. Harris, \textit{Algebraic geometry}, Graduate Texts in Mathematics \textbf{133}, Springer-Verlag, New York, 1992.


\bibitem{Hartshorne62} R. Hartshorne, \textit{Complete intersections and connectedness}, Amer. J. Math. \textbf{84} (1962), 497--508.


\bibitem{Hartshorne} R. Hartshorne, \textit{Algebraic geometry}, Graduate Texts in Mathematics \textbf{52}, Springer-Verlag, New York-Heidelberg, 1977. 



\bibitem{HSV82}
J.~Herzog, A.~Simis, and W.~V. Vasconcelos, \textit{Approximation complexes of blowing-up rings}, J. Algebra \textbf{74} (1982), 466--493.

\bibitem{HSV83}
J.~Herzog, A.~Simis, and W.~V. Vasconcelos, \textit{Approximation complexes of blowing-up rings. {II}}, J. Algebra \textbf{82} (1983), 53--83.
  
 
  

\bibitem{HR86} C.~Huneke and M.~Rossi, \textit{The dimension and components of symmetric algebras}, J. Algebra \textbf{98} (1986), 200--210.  
  





\bibitem{Jouanolou96} J.-P. Jouanolou, \textit{R\'esultant anisotrope, compl\'ements et application}, Electron. J. Combin. \textbf{3} (1997), research paper 2, approx. 91 pp.


\bibitem{KM20} Y. Kim and V. Mukundan, \textit{Equations defining certain graphs}, Michigan Math. J. \textbf{69} (2020), 675--710.


\bibitem{Kunz} E. Kunz, \textit{K\"ahler Differentials}, Advanced Lectures in Mathematics, Friedr. Vieweg \& Sohn, Braunschweig, 1986. 


  
\bibitem{KPU17} A.~R. Kustin, C. Polini, and B.~Ulrich, \textit{The equations defining blowup algebras of height three Gorenstein ideals}, Algebra Number Theory \textbf{11} (2017), 1489--1525.  
  
\bibitem{KPU Bigraded Structures} A.~R. Kustin, C. Polini, and B.~Ulrich, \textit{The bi-graded structure of symmetric algebras with applications to Rees rings}, J. Algebra \textbf{469} (2017), 188--250.    
  

\bibitem{Morey96} S. Morey, \textit{Equations of blowups of ideals of codimension two and three}, J. Pure Appl. Algebra \textbf{109} (1996), 197--211.   

\bibitem{MU96} S. Morey and B. Ulrich, \textit{Rees algebras of ideals with low codimension}, Proc. Amer. Math. Soc. \textbf{124} (1996), 3653--3661.
  
\bibitem{Nguyen14} P.~H.~Lan Nguyen, \textit{On Rees algebras of linearly presented ideals}, J. Algebra \textbf{420} (2014), 186–-200.

\bibitem{Nguyen17} P.~H.~Lan Nguyen, \textit{On Rees algebras of linearly presented ideals in three variables}, J. Pure Appl. Algebra \textbf{221} (2017), 2180–-2198.
 

  
\bibitem{SUV93} A.~Simis, B.~Ulrich, and W.~V. Vasconcelos, \textit{Jacobian dual fibrations}, Amer. J. Math. \textbf{115} (1993), 47--75. 
  


\bibitem{SUV97} A.~Simis, B.~Ulrich, and W.~V. Vasconcelos, \textit{Tangent star cones}, J. Reine Angew. Math. \textbf{483} (1997), 23--59.


\bibitem{SUV03} A.~Simis, B.~Ulrich, and W.~V. Vasconcelos, \textit{Rees algebras of modules}, Proc. London Math. Soc. \textbf{87} (2003), 610--646.


\bibitem{SUV12} A.~Simis, B.~Ulrich, and W.~V. Vasconcelos, \textit{Tangent algebras}, Trans. Amer. Math. Soc. \textbf{364} (2012), 571--594.

  
\bibitem{UV93} B.~Ulrich and W.~V. Vasconcelos, \textit{The equations of {R}ees algebras of ideals with linear presentation}, Math. Z. \textbf{214} (1993), 79--92.

\bibitem{Vasconcelos67} W.~V. Vasconcelos, \textit{Ideals generated by $R$-sequences}, J. Algebra \textbf{6} (1967), 309--316.

\bibitem{Vasconcelos78} W.~V. Vasconcelos, \textit{On the homology of $I/I^2$}, Comm. Algebra \textbf{6} (1978), 1801--1809.
  
\bibitem{Vasconcelos91} W.~V. Vasconcelos, \textit{On the equations of {R}ees algebras}, J. Reine Angew. Math. \textbf{418} (1991), 189--218.
  
\bibitem{VasconcelosBook} W.~V. Vasconcelos, \textit{Arithmetic of blowup algebras}, London Math. Soc. Lecture Note Ser. \textbf{195}, Cambridge University Press, Cambridge, 1994. 

\bibitem{VasconcelosBook2} W.~V. Vasconcelos, \textit{Integral closure}, Springer Monographs in Mathematics, Springer-Verlag, Berlin, 2005. 




 \bibitem{Weaver23} M. Weaver, \textit{On Rees algebras of ideals and modules over hypersurface rings}, J. Algebra \textbf{636} (2023), 417--454.

\bibitem{Weaver24} M. Weaver, \textit{The equations of Rees algebras of height three Gorenstein ideals in hypersurface rings}, J. Commut. Algebra \textbf{16} (2024), 123--149.



\bibitem{Yoshino85}
Y. Yoshino, \textit{Codimension of Jacobian ideals and ($R_n$) conditions for complete intersections}, Hiroshima Math. J. \textbf{15} (1985), 663--667.

 
\end{thebibliography}
\end{document}